\documentclass[a4paper,notitlepage]{article}
\usepackage[latin1]{inputenc}
\usepackage[T1]{fontenc}
\usepackage{amsmath}
\usepackage{amsfonts}
\usepackage{amssymb}
\usepackage{latexsym}
\usepackage{stmaryrd}
\usepackage[english]{babel}
\usepackage{tikz}
\usepackage{amsthm}
\usepackage{hyperref}
\usepackage{dsfont}
\usepackage{pst-eucl}
\usepackage{geometry}
\usepackage{tikz}
\tikzset { domaine/.style 2 args={domain=#1:#2} }
\tikzset{
 xmin/.store in=\xmin, xmin/.default=-3, xmin=-3,
 xmax/.store in=\xmax, xmax/.default=3, xmax=3,
 ymin/.store in=\ymin, ymin/.default=-3, ymin=-3,
 ymax/.store in=\ymax, ymax/.default=3, ymax=3,
 }

\newtheorem{theo}{Theorem}[section]

\newtheorem{prop}{Proposition}[section]
\newtheorem{co}{Corollary}[section]
\newtheorem{defi}{Definition}[section]
\newtheorem{lem}{Lemma}[section]

\newcommand{\dx}{\, \text{\textnormal{d}}}
\newcommand{\R}{\mathbb{R}}
\newcommand{\N}{\mathbb{N}}

\newcommand{\ve}{\varepsilon}
\newcommand{\eps}{\varepsilon}
\newcommand{\leb}{\mathcal{L}}

\renewcommand{\phi}{\varphi}

\renewcommand{\bar}{\overline}
\renewcommand{\tilde}{\widetilde}
\newcommand{\Lip}{\operatorname{Lip}}
\newcommand{\ddt}{\dfrac{\text{\textnormal{d}}}{\text{\textnormal{d}t}}}
\renewcommand{\O}{\mathcal{O}}
\newcommand{\gammato}{\xrightarrow{\Gamma}}

\newcommand{\supp}{\operatorname{supp}}

\newcommand{\grando}{\mathop{O}}

\newcommand{\id}{\text{id}}
\newcommand{\wto}{\rightharpoonup}

\title{The Monge problem with vanishing gradient penalization:\\ vortices and asymptotic profile}
\author{Luigi De Pascale\thanks{Dipartimento di Matematica, Universit\`a di Pisa, Largo Bruno Pontecorvo 5 - 56127 Pisa, Italy (\url{depascal@dm.unipi.it})}, Jean Louet\thanks{Laboratoire de Mathématiques d'Orsay, Universit\'e Paris-Sud, CNRS, Universit\'e Paris-Saclay, 91405 Orsay, France (\url{jean.louet@math.u-psud.fr}, \url{filippo.santambrogio@math.u-psud.fr})}, Filippo Santambrogio\footnotemark[2]}
\date{\today}

\begin{document}
\maketitle

\begin{abstract}
\noindent We investigate the approximation of the Monge problem (minimizing $\int_\Omega |T(x)-x|\dx\mu(x)$ among the vector-valued maps $T$ with prescribed image measure $T_\#\mu$) by adding a vanishing Dirichlet energy, namely $\eps\int_\Omega |DT|^2$. We study the $\Gamma$-convergence as $\eps\to 0$, proving a density result for Sobolev (or Lipschitz) transport maps in the class of transport plans. In a certain two-dimensional framework that we analyze in details, when no optimal plan is induced by an $H^1$ map, we study the selected limit map, which is a new ``special'' Monge transport, possibly different from the monotone one, and we find the precise asymptotics of the optimal cost depending on $\eps$, where the leading term is of order~$\eps|\log\eps|$.

\vspace{0.3cm} \noindent {\bf Keywords.} Optimal transport, Monge problem, monotone transport, $\Gamma$-convergence, density of smooth~maps

\vspace{0.3cm} \noindent{\bf AMS subject classification.} 49J30, 49J45

\end{abstract}

\section{Introduction}

This paper investigates the following minimization problem: given $\mu$, $\nu$ two - smooth enough - probability densities on $\R^d$ with $\mu$ supported in a domain $\Omega$, we study
$$ \inf\left\{ J_\eps(T) \,:\, T_\#\mu = \nu \right\} \qquad \text{where} \qquad J_\eps(T) = \int_\Omega |T(x)-x|\dx\mu(x) + \eps\int_\Omega |DT|^2 \dx x. $$
Here $\eps$ is a vanishing positive parameter, $|\cdot|$ is the usual Euclidean norm on $\R^d$ or $M_d(\R)$, $DT$ denotes the Jacobian matrix of the vector-valued map~$T$ and $T_\#\mu$ is the image measure of $\mu$ by $T$ (defined by $T_\#\mu (B)= \mu(T^{-1}(B))$ for any Borel set $B \subset \R^d$). We aim to understand the behavior of the functional $J_\eps$ in the sense of $\Gamma$-convergence and to characterize the limits of  the minimizers $T_\eps$.\bigskip

If we do not take in to account this gradient penalization, we recover the classical optimal transport problem originally proposed by Monge in the 18th century \cite{monge}. For this problem, the particular constraint $T_\#\mu=\nu$ makes the existence of minimizers quite difficult to obtain by the direct method of the calculus of variations; when the Euclidean distance is replaced with nicer functions (usually strictly convex with respect to the difference $x-T(x)$), strong progresses have been realized by Kantorovich in the '40s \cite{kant, kant2} and Brenier in the late '80s \cite{brenierpolar}.
 In the Monge's case of the Euclidean distance, the existence results have been shown more recently by several techniques: we just mention the first approach by Sudakov \cite{s} (later completed by Ambrosio \cite{a}) and the differential equations methods by Evans and Gangbo \cite{eg}, for the case of the Euclidean norm. More recently, this has been generalized to uniformly convex norms by approximation \cite{cfmc,tw}; finally,~\cite{cdp} generalizes this result for any generic norm in $\R^d$. We refer to \cite{vil} for a complete overview of the optimal transport theory, and to the lecture notes~\cite[Section~3.1]{sln} or again~\cite{a} for the particular case of the Monge problem with Euclidean norm.

Adding a Sobolev-like penalization is very natural in many applications, for instance in image processing, when these transport maps could model transformations in the space of colors, which are then required to avoid abrupt variations and discontinuities (see \cite{fprpa}). Also, these problems appear in mechanics (see \cite{gm}) and computer-science problems \cite{lps} , where one looks for maps with minimal distance distortion (if possible, local isometries with prescribed image measure). 

However, in this paper we want to concentrate on the mathematical properties of this penalization. Notice first that this precisely allows to obtain very quickly the existence of optimal maps, since we get more compactness and we can this time use the direct method of the calculus of variations (see Prop.\@~\ref{directmethod} below; actually, this is trickier in the case where the source measure is not regular, {\it cf.}\@~\cite{l} and \cite[Chapter 1]{lphd}; this requires to use the theory developed in \cite{bbs}). The motivation for this vanishing penalization comes from the particular structure of the Monge problem. It is known that the minimizers of $J_0$ (that we will denote by $J$ in the following) are not unique and are exactly the transport maps from $\mu$ to $\nu$ which also send almost any source point $x$ to a point $T(x)$ belonging to the same {\it transport ray} as $x$ (see below the precise definition). Among these transport maps, selection results are particularly useful, and approximating with strictly convex transport costs $c_\eps(x,y)=|x-y|+\eps|x-y|^2$ brings to the {\it monotone transport}, {\it i.e.}~the unique transport map which is non-decreasing along each transport ray. About this special transport map, some regularity properties have been proven: continuity in the case of regular measures with disjoint and convex supports in the plane \cite{fgp}, uniform estimates on an approximating sequence under more general assumptions \cite{lsw}. The question that we propose here is to know which of these transport maps is selected by the approximation through the gradient penalization; it is natural to wonder whether this map is again the monotone one, and, more in general, to expect nice regularity properties for the selected map. \bigskip

In this paper, we analyze the behavior of $J_\eps$ when $\eps$ vanishes in the sense of the $\Gamma$-convergence (see \cite{braides} for the definitions and well-known results about this notion). First of all, we show the ``zeroth-order'' $\Gamma$-convergence of $J_\eps$ to the transport energy~$J$; although straightforward, the proof requires a non-trivial density result, namely the density of the set of Sobolev maps $T \in H^1(\Omega)$ sending $\mu$ to $\nu$ among the set of transport maps. This result looks natural and can of course be used in other contexts. Even if a similar statement was already present in \cite{BreGan,Schnir}, it was not stated in the same spirit as the formulation that we give here below. Moreover, we provide a significantly different proof, which recalls the proof of the density of transport maps into transport plans of~\cite[Section 1.5]{sln}.\smallskip

\begin{theo} \label{denssobolev} Let $\Omega$, $\Omega'$ be two Lipschitz polar domains of $\R^d$. Let $\mu \in \mathcal{P}(\Omega)$, $\nu \in \mathcal{P}(\Omega')$ be two probability measures, both absolutely continuous with respect to the Lebesgue measure, with densities $f$, $g$; assume that $f$, $g$ belong to $C^{0,\alpha}(\bar{\Omega})$, $C^{0,\alpha}(\bar{\Omega'})$ for some $\alpha > 0$ and are bounded from above and below by positive constants. Then the set
$$ \left\{ T \in \Lip(\Omega) \,:\, T_\#\mu = \nu \right\}$$
is non-empty, and is a dense subset of the set
$$ \left\{ T : \Omega \to \Omega' \,:\, T_\#\mu = \nu \right\} $$
endowed with the norm $||\cdot||_{L^2(\Omega)}$.

As a consequence, under these assumptions on $\Omega$, $\Omega'$, $\mu$, $\nu$, we have $J_\eps \gammato J$ as $\eps \to 0$. \end{theo}

\noindent The definition of ``Lipschitz polar domain'' is given below in Definition~\ref{lipdom}; it is a large class of star-shaped domains having Lipschitz boundary. We do not claim this assumption on $\Omega,\Omega'$ to be sharp, and introduced this definition essentially for technical reasons (see the proof in paragraph~\ref{proofdensity}).

Notice also that this density result would be completely satisfactory if the assumptions on the measures in order to get density of Lipschitz maps were the same as to get existence of at least one such a map. The assumptions that we used are likely not to be sharp, but are the most natural one if one wants to guarantee the existence of $C^{1,\alpha}$ transport maps (and it is typical in regularity theory that Lipschitz regularity is not easy to provide, whereas H\"older results work better; notice on the contrary that a well-established $L^p$ theory is not available in this framework). Anyway, from the proof that we give in Section~2, it is clear that we are not using much more than the simple existence of Lipschitz maps. It is indeed an interesting fact, already observed in  \cite[Section 1.5]{sln} that density results for a class $D$ in the set of transport plans or maps are essentially proven under the same natural assumptions on~$\mu$ and~$\nu$ which guarantee that there exists admissible transport plans or maps in $D$: here the set of smooth maps is dense, under assumptions essentially corresponding to those we need for the existence of smooth transport maps (exactly as it happens that transport maps are dense in transport plans provided $\mu$ has no atoms, which is exactly the standard assumption to guarantee that transport maps do exist). \smallskip

Now that the ``zeroth-order'' $\Gamma$-convergence is proven, the natural question which arises concerns the behavior of the remainder with respect to the order $\eps$: denoting by $W_1$ the optimal value of the Monge problem
, we need to consider $(J_\eps-W_1)/\eps$. The result is the most natural that we expect (and the proof is this time almost trivial):
\begin{equation} \frac{J_\eps(T_\eps)-W_1}{\eps} \gammato \mathcal H \label{ordereps} \end{equation}
$$ \text{where} \qquad \mathcal H(T) = \left\{ \begin{array}{l} \displaystyle\int_\Omega |DT|^2 \qquad \text{if } T \in \O_1(\mu,\nu)\cap H^1(\Omega) \\[2mm] +\infty \qquad \text{otherwise} \end{array} \right. $$
where we have denoted by $\O_1(\mu,\nu)$ the set of optimal maps for the Monge problem. Notice that this result allows immediately to build some examples of measures $\mu$, $\nu$ for which the minimal value of the function $\mathcal H$ is not attained by the monotone transport map from $\mu$ to $\nu$, thanks to suitable analysis of the minimization of the $H^1$-norm among the set of transport maps on the real line (which has been very partially treated in~\cite{ls}). \smallskip

The convergence \eqref{ordereps} gives immediately a first order approximation (meaning $\inf J_\eps = W_1 + \eps \inf \mathcal H + o(\eps)$) provided that $\inf \mathcal H \neq +\infty$, {\it i.e.}~when there exists at least one map which minimizes the Monge cost and belongs to the Sobolev space $H^1(\Omega)$. If no such map exists, the only information given by~\eqref{ordereps} is that the order of convergence of the minimal value is strictly smaller than~$1$, meaning that $(J_\eps(T_\eps)-W_1)/\eps \to +\infty$ for any family of maps $(T_\eps)_\eps$; in this case, both the order of convergence of $\inf J_\eps$ to $W_1$ and the selected map as $\eps \to 0$ are unknown. \bigskip

The rest of this paper is devoted to the complete study of this case in a particular example, which has very interesting properties with respect to these questions. We take as source domain the quarter disk in the plane $\Omega = \left\{ x = (r,\theta) \,:\, 0< r < 1, 0 < \theta < \pi/2 \right\}$ and as target domain an annulus located between two regular curves in polar coordinates $\Omega'= \left\{ x = (r,\theta) \,:\, R_1(\theta)< r < R_2(\theta), 0 < \theta < \pi/2 \right\}$; we endow these domains with two regular densities $f$, $g$ satisfying the condition:
$$ \text{for any }\theta, \quad \int_0^1 f(r,\theta) \,r\dx r = \int_{R_1(\theta)}^{R_2(\theta)} g(r,\theta) \, r \dx r $$
which means that the mass (with respect to $f$) of the segment with angle $\theta$ joining the origin to the quarter unit circle is equal to the mass (with respect to $g$) of the segment with same angle joining the two boundaries of $\Omega'$. The situation is described by Figure~\ref{figomega}.

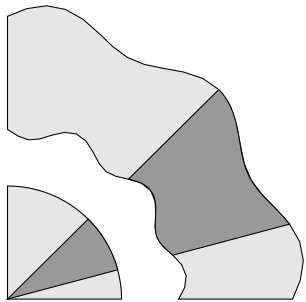
\begin{figure}[h]
\begin{center}
\begin{tikzpicture}[scale=1.5]

\filldraw[color=gray!20,draw=black] (0,0) -- (1,0) -- plot [domain=0:90] (\x:1) -- cycle ;

\filldraw[color=gray!20,draw=black] (1.5,0) -- plot [domain=0:90] (\x:{1.5+0.08*sin(12*\x)}) -- (0,2.5) -- plot [domain=0:90] (90-\x:{2.5+0.12*sin(10*\x)}) -- cycle ;

\filldraw[color=gray!80,draw=black] (15:0) -- (15:1) -- plot [domain=15:45] (\x:1) -- cycle;
\filldraw[color=gray!80,draw=black] (15:1.5) -- plot [domain=15:45] (\x:{1.5+0.08*sin(12*\x)}) -- (45:2.5) -- plot [domain=0:30] (45-\x:{2.5+0.12*sin(10*(45-\x))}) -- cycle;

\end{tikzpicture}
\caption{The domains $\Omega$, $\Omega'$. The two filled sectors have same mass.} \label{figomega}
\end{center}
\end{figure}

Under these assumptions, the transport rays are supported by the lines starting from the origin and the optimal transport maps for the Monge problem send each $x \in \Omega$ onto a point $T(x)$ with same angle; in other words, any optimal $T$ is written as
$$ T(x) = \phi(x)\frac{x}{|x|} $$
for some scalar function $\phi$ with $\inf \phi>0$. In particular, such a map~$T$ must present a singularity at the origin: if we write $\phi$ in polar coordinates, and if $\phi(r,\theta)$ has some regularity as $r\to 0$, the point~$0$ is sent by $T$ on the whole curve $\theta \mapsto \phi(0,\theta)$. It is then elementary to prove that $\int |DT|^2=+\infty$: we thus precisely face a case where $\mathcal H \equiv +\infty$ and, as we have seen above, the asymptotics $\inf J_\eps = W_1+O(\eps)$ is impossible, while the selected map as $\eps\to 0$ is unknown {\it a priori}.

In order to guess the behavior of $J_\eps$ as $\eps\to 0$, the preliminary analysis we propose is the following: starting from an optimal $T$ for the Monge problem, which is given by $T(x)=\phi(x)\cdot x/|x|$, we assume that $T$ is regular enough away from the origin. Then, a natural construction of a map $T_\eps$ having a finite Sobolev norm and which is close from $T$ consists in modifying $T$ only around the origin, and then in making it regular, while keeping for $T_\eps$ the constraint of sending $\mu$ onto $\nu$. This is possible thanks to the Dacorogna and Moser's result (see \cite{dm}, and below Theorem~\ref{appdacmos} which is the statement we actually use is this paper), which allows, given two regular enough measures, to send the first one to the second one through a Lipschitz diffeomorphism. From this perturbation, it then turns out (the formal computations which lead to it are presented in Paragraph~\ref{sectheu}) that
$$ J_\eps(T_\eps) = W_1+\eps|\log\eps| \frac{1}{3} ||\phi(0,\cdot)||_{H^1(0,\pi/2)}^2 + \grando\limits_{\eps\to 0}(\eps)  $$
when $\phi(x)=\phi(r,\theta)$ in polar coordinates. This expansion suggest the following phenomena:
\begin{itemize}
\item the leading term of the asymptotics of $J_\eps-W_1$ as $\eps\to 0$ has order $\eps|\log\eps|$;
\item the main ``rest'' (namely, $(J_\eps(T_\eps)-W_1)/(\eps|\log\eps|)$) involves only the behavior of $T$ around the origin;
\item more precisely, among the optimal transport maps for the Monge problem, the selected maps by the approximation seem to be those which send the origin to the curve $\theta\mapsto \phi(0,\theta)$ having the smallest possible one-variable Sobolev norm. \smallskip
\end{itemize}

The main result of this paper and its consequences allow to prove that this description of the behavior of $J_\eps$ and of its minimizers, as $\eps\to 0$, actually holds. As the above analysis suggests, we introduce the following notations:
\begin{itemize}
\item $\Phi$ is the one-variable function which realizes the smallest Sobolev norm, among the curves (in polar coordinates) which fully belong to the target domain $\bar{\Omega'}$, and $K$ is the square of its Sobolev norm:
\begin{equation} \Phi = \operatorname{argmin}\,\left\{ \int_0^{\pi/2} (\phi^2+\phi'^2) \,:\, R_1(\theta) \leq \phi(\theta) \leq R_2(\theta) \right\}\label{Phi} \end{equation}
\begin{equation} \text{and} \qquad K = ||\Phi||_{H^1(0,\pi/2)}^2 = \min\,\left\{ \int_0^{\pi/2} (\phi^2+\phi'^2) \,:\, R_1(\theta) \leq \phi(\theta) \leq R_2(\theta) \right\}\label{K} \end{equation}
(the existence and uniqueness of $\Phi$ are easy to obtain from convexity and semi-continuity properties of the problem which defines it);
\item for any $\eps>0$, the functional $F_\eps$ is defined, on the set of transport maps from $\mu$ to $\nu$, as
$$ F_\eps = \frac 1\eps\left(J_\eps - W_1 - \frac{K}{3}\eps|\log\eps|\right) $$
\end{itemize}
Our main result reads then as follows:

\begin{theo} \label{maintheo} Assume that the functions $R_1$, $R_2$ are Lipschitz on the interval $(0,\pi/2)$, and that $R_1'$ has finite total variation. Assume that $f$, $g$ are both Lipschitz and bounded from above and from below by positive constants on $\bar{\Omega}$, $\bar{\Omega'}$. We define the functional $F$ as:
\begin{itemize}
\item if $T$ does not belong to $\O_1(\mu,\nu)$, then $F(T)=+\infty$;
\item if $T$ belongs to $\O_1(\mu,\nu)$ and if, for any $x = (r,\theta) \in \Omega$, $T(x)=\phi(r,\theta)\dfrac{x}{|x|}$, we have
$$  F(T) = \int_0^1 \frac{||\phi(r,\cdot)||_{H^1(0,\pi/2)}^2-K}{r} \dx r + \int_0^1 ||\partial_r\phi(r,\cdot)||_{L^2(0,\pi/2)}^2 \, r \dx r $$
\end{itemize}
Then the following properties hold:
\begin{enumerate}
\item For any family of maps $(T_\eps)_\eps$ such that $(F_\eps(T_\eps))_\eps$ is bounded, there exists a sequence $\eps_k \to 0$ and a map $T$ such that $T_{\eps_k} \to T$ in~$L^2(\Omega)$.
\item There exists a constant $C$, depending only on the domains $\Omega$, $\Omega'$ and on the measures $f$, $g$, so that, for any family of maps $(T_\eps)_{\eps>0}$ with $T_\eps \to T$ as $\eps \to 0$ in $L^2(\Omega)$, we~have
$$ \liminf\limits_{\eps \to 0} F_\eps(T_\eps) \geq F(T)-C $$
\item Moreover, there exists at least one family $(T_\eps)_{\eps > 0}$ such that $(F_\eps(T_\eps))_\eps$ is indeed bounded. \end{enumerate} \end{theo}

\noindent Let us state also here the following property of the functional $F$, which allows, together with Theorem~\ref{maintheo}, to achieve the qualitative description of any limit of minimizers of $J_\eps$:

\begin{prop} Let $T$ be a map such that $F(T)<+\infty$. Then, denoting by $T(x)=\phi(r,\theta)\cdot x/|x|$, the function
$$ r \mapsto \phi(r,\cdot) $$
is continuous from $[0,1]$ to $L^2(0,\pi/2)$ and satisfies $\phi(0,\cdot)=\Phi$. \end{prop}

\noindent The meaning of this proposition is that any map having finite energy $F$ (and, thanks to Theorem~\ref{maintheo}, this is the case of any map $T$ which is limit of minimizers of $J_\eps$) must, as we guessed, send the origin to the curve of $\Omega'$ which is defined by $\Phi$. In particular, depending on the shape of $\Omega'$, it may appear that such a map is different from the monotone transport: this is the case as soon as the function $\Phi$ is not identically equal to $R_1$, the ``lower boundary'' of the domain $\Omega'$. \smallskip

Let us make some comments about Theorem~\ref{maintheo}. This result does not give precisely the $\Gamma$-limit of the functional $F_\eps$ but only a lower bound on the $\Gamma$-liminf; we conjecture that $F_\eps \gammato F-C$, where~$C$ is a suitable constant associated to the domains. However, this result is enough to obtain some important consequences on the behavior of the approximation. Indeed, the fact that there exists a family $(T_\eps)_\eps$ such that $(F_\eps(T_\eps))_\eps$ is bounded implies that this is the case when we take a sequence $T_\eps$ of minimizers. In in this case we necessary have, up to subsequences, $T_\eps \to T$ with $F(T)<+\infty$, which implies then that $T$ maps the origin onto the curve $\Phi$. Also, the asymptotic behavior which appears is precisely
$$ \inf J_\eps = W_1 + \frac{K}{3}\eps|\log\eps| + \grando\limits_{\eps\to 0}(\eps) $$
where $K = ||\Phi||_{H^1}^2$, the minimal (square of the) Sobolev norm of the curves valued in $\bar{\Omega'}$. \bigskip

We also notice that the $\eps|\log\eps|$ term in the energy is due to the blow-up of the Sobolev norm at a single point. This fact suggests a formal but deep analogy with the Ginzburg-Landau theory (see for instance \cite{bbh}), where one looks at the minimization of
\begin{equation}\label{GLbbh}
 u \mapsto \frac{1}{\eps^2} \int_\Omega (1-|u|^2)^2 + \int_\Omega |\nabla u|^2
 \end{equation}
with given boundary conditions on $u$ (assuming $|u|=1$ on $\partial \Omega$). Here, as $\ve\to 0$, the two terms are contradictory in the functional, as one requires $|u|\approx 1$ and the other $u\in H^1$. For instance, in the case where $\Omega$ is a ball and $u$ is constrained to be the identity on $\partial\Omega$, it is not possible to select an $H^1$ vector field, of unit norm, with the required boundary datum. In our case, the situation is similar. The two contradictory phenomena are the fact that $T$ has to preserve the transport rays (to optimize the Monge problem) and that it has a finite Sobolev norm; this also leads to the creation of an explosion (that we could call a vortex; here we have explosion at the origin, which is sent to a whole curve belonging to the target domain). The excess of order $\eps|\log\eps|$ is a common feature of the two problems.

If we want to push the analogy to a further level, we can recall the main precise result of the Ginzburg-Landau analysis of \cite{bbh}: the minimal value in \eqref{GLbbh} is asymptotically equivalent to $2\pi d_0|\log\eps|$, where $d_0$ is the topological degree of the map $u:\partial\Omega\to\mathbb S^1$, {\it i.e.}~it depends on the boundary condition. In our problem, we also identify the coefficient in front of the logarithmic cost, which is given by the constant $K/3$ in \eqref{K}: it depends on $\Omega'$, which is indeed part of the ``boundary condition'' ({\it i.e.}~of the constraint $T_\#\mu=\nu$ of our problem; notice that the fact that $T$ should send $\Omega$ into $\Omega'$ is usually called {\it second boundary condition} in the Monge-Amp\`ere community). Then, after the logarithmic term, both in the Ginzburg-Landau model and here there is a lower order analysis. In the Ginzburg-Landau problem this lets a functional based on the position where the vertices appear, while here we are only able to give some bounds on the $\Gamma$-limit of the re-scaled functional: however, the estimates for the $\Gamma$-liminf and the $\Gamma$-limsup only differ by an additive constant, and are enough to provide a satisfactory qualitative behavior of the optimal structure. This allows to find the limit profile of the optimal maps $T_\ve$ in polar coordinates, close to the vortex $r=0$. 

\paragraph{Plan of the paper.} Section 1 collects general notations and well-known facts about the Monge problem, basic notions of $\Gamma$-convergence and some useful and elementary results about the optimal transport with gradient term. In Section 2, we prove the density of the Sobolev transport maps among the transport maps and state the results of $\Gamma$-convergence with order $0$ and with order $1$ of $(J_\eps)_\eps$ for generic (regular) domains and measures. In Section 3, we study precisely the example of $\eps|\log\eps|$ approximation in the above framework; the main results and its interpretations are given in Paragraph 3.3, following a formal computation that we present in Paragraph 3.2. Section 4 is completely devoted to the rigorous proof of the main result of this paper (Theorem \ref{maintheo}).

\paragraph{Acknowledgements.} This work has been developed thanks to some visits of the second and third authors to the Universit\`a di Pisa, which have been possible thanks to the support of ANR project ANR-12-BS01-0014-01 GEOMETRYA and the PGMO project MACRO, of EDF and Fondation Math\'ematique Jacques Hadamard. This material is also based upon work supported by the National Science Foundation under Grant No.\@ 0932078~000, while the second author was in residence at the Mathematical Science Research Institute in Berkeley, California, during the Fall term 2013.
The research of the first author is part of the project 2010A2TFX2 {\it``Calcolo delle Variazioni''}, financed by the Italian Ministry of Research
 and is partially financed  by the {\it ``Fondi di ricerca di ateneo''}  of the  University of Pisa. 

\section{Preliminary notions}

\subsection{Known facts about the Monge problem}

In this section, we recall some well-known facts and useful tools about the optimal transportation problem with the Monge cost $c(x,y)=|x-y|$, where $|\cdot|$ is the Euclidean norm on $\R^d$. The proofs and many details can be found in \cite{a} and \cite[Section~3.1]{sln}.

Let $\Omega$, $\Omega'$ be two bounded open sets on $\R^d$, $\mu \in \mathcal{P}(\Omega)$ and $\nu \in \mathcal{P}(\Omega')$; assume that $\mu$ is absolutely continuous with respect to the Lebesgue measure with density $f$. Then we set
$$ W_1(\mu,\nu) = \min\left\{ \int_\Omega |T(x)-x| \dx \mu(x) \;:\; T : \Omega \to \R^d,\, T_\#\mu = \nu \right\} $$
the minimal value of the Monge transport cost from $\mu$ to $\nu$, and
$$ \O_1(\mu,\nu) = \left\{ T : \Omega \to \R^d \;:\; T_\# \mu = \nu \; \text{and} \; \int_\Omega |T(x)-x|\dx\mu(x) = W_1(\mu,\nu) \right\} $$
the set of optimal transport maps for the Monge cost (if there is no ambiguity, we will simply use the notations $W_1$ and $\O_1$). Notice that the above Monge problem is a particular issue of its Kantorovich formulation, {\it i.e.}
$$\min\left\{ \int_{\Omega\times\Omega'} |y-x| \dx \gamma(x,y) \;:\; \gamma\in \mathcal{P}(\Omega\times\Omega'): (\pi_x)_\#\gamma=\mu, \, (\pi_y)_\#\gamma= \nu \right\},$$
and the minimal value of this problem coincides with $W_1(\mu,\nu)$.

\begin{theo}[Duality formula for the Monge problem] We have the equality
$$ W_1(\mu,\nu) = \sup\left\{ \int_{\Omega'} u(y)\dx\nu(y) - \int_\Omega u(x)\dx\mu(x) : u \in \Lip_1(\R^d) \right\} $$
where $\Lip_1(\R^d)$ denotes the set of $1$-Lipschitz functions $\R^d \to \R$. The optimal functions $u$ are called \textnormal{Kantorovich potentials}. Moreover, if $\Omega$ is a connected set and if the set $\{f>0\}\setminus\Omega'$ is a dense subset of $\Omega$, then the restriction of the optimal $u$ to $\supp\mu\cup\supp\nu$ is unique, up to an additive constant. \end{theo}

\noindent As a direct consequence of the duality formula, if $T:\Omega \to \R^d$ sends $\mu$ to $\nu$ and $u \in \Lip_1(\R^d)$, we have the equivalence:
$$ \left\{ \begin{array}{l} T \in \O_1(\mu,\nu) \\ u \text{ is a Kantorovich potential} \end{array} \right. \iff \text{for } \mu\text{-a.e. } x\in \Omega,\; u(T(x))-u(x) = |T(x)-x| $$
We now introduce the following crucial notion of {\it transport ray}:

\begin{defi}[Transport rays] Let $u$ be a Kantorovich potential and $x \in \Omega$, $y \in \Omega'$. Then:
\begin{itemize}
\item the open oriented segment $(x,y)$ is called \textnormal{transport ray} if $u(y)-u(x)=|y-x|$;
\item the closed oriented segment $[x,y]$ is called \textnormal{maximal transport ray} if any point of $(x,y)$ is contained into at least one transport ray with same orientation as $[x,y]$, and if $[x,y]$ is not strictly included in any segment with the same property.
\end{itemize} \end{defi}

\begin{prop}[Geometric properties of transport rays] The set of maximal transport rays does not depend on the choice of the Kantorovich potential $u$ and only depends on the source and target measures~$\mu$ and $\nu$. Moreover:
\begin{itemize}
\item any intersection point of two different maximal transport rays is an endpoint of these both maximal transport rays;
\item the set of the endpoints of all the maximal transport rays is Lebesgue-negligible.
\end{itemize} \end{prop}

\noindent These notions allow to prove the existence and to characterize the optimal transports:

\begin{prop}[Existence and characterization of optimal transport maps] The solutions of the Monge problem exist and are not unique; precisely, a map $T$ sending $\mu$ to $\nu$ is optimal if and only if:
\begin{itemize}
\item for a.e.~$x \in \Omega$, $T(x)$ belongs to the same maximal transport ray as $x$;
\item the oriented segment $[x,T(x)]$ has the same orientation as this transport ray.
\end{itemize} \end{prop}

\noindent We finish by recalling that, among all maps in $\O_1(\mu,\nu)$, there is a special one which has received much attention so far, and by quickly reviewing its properties.

\begin{prop}[The monotone map and a secondary variational problem] If $\mu$ is absolutely continuous, there exists a unique transport map $T$ from $\mu$ to $\nu$ such that, for each maximal transport ray $S$, $T$ is non-decreasing from the segment $S \cap \Omega$ to the segment $S \cap \Omega'$ (meaning that if $x,x' \in \Omega$ belong to the same transport ray, then $[x,x']$ and $[T(x),T(x')]$ have the same orientation).Moreover, $T$ solves the problem
$$ \inf\left\{ \int_{\Omega} |T(x)-x|^2 \dx\mu(x) : T \in \O_1(\mu,\nu) \right\}.$$
\end{prop}

\noindent Notice that this solution is itself obtained as limit of minimizers of a perturbed variational problem, namely
$$ \inf\left\{ \int_\Omega |T(x)-x|\dx\mu(x) + \eps\,\int_\Omega |T(x)-x|^2\dx\mu(x) \,:\, T_\#\mu = \nu \right\} $$
This very special transport map is probably one of the most natural in $\O_1(\mu,\nu)$ and one of the most regular. In particular, under some assumptions on the densities $f$, $g$ and their supports $\Omega$, $\Omega'$ (convex and disjoint supports in the plane and continuous and bounded by above and below densities), this transport has also been shown to be continuous \cite{fgp}, and \cite{lsw} gives also some regularity results for the minimizer $T_\eps$ of an approximated problem where $c$ is replaced with $c_\eps(x,y) = \sqrt{\eps^2+|x-y|^2}$ (local uniform bounds on the eigenvalues of the Jacobian matrix $T_\eps$).

\subsection{Tools for optimal transport with gradient penalization}

\paragraph{Existence of solutions.} We begin by showing the existence of solutions for the penalized problem with an elementary proof:

\begin{prop} \label{directmethod} Let $\Omega$ be a bounded open set of $\R^d$ with Lipschitz boundary. Let $\mu \in \mathcal{P}(\Omega)$ be absolutely continuous 
with density $f$, and $\nu \in \mathcal{P}(\R^d)$ such that the set
$$ \left\{ T \in H^1(\Omega) \,:\, T_\#\mu = \nu \right\} $$
is non-empty. Then for any $\eps > 0$, the problem
\begin{equation} \inf\left\{ \int_\Omega |T(x)-x| f(x) \dx x + \eps \int_\Omega |DT(x)|^2 \dx x \,:\, T \in H^1(\Omega), \, T_\#\mu=\nu \right\} \label{otwithgrad} \end{equation}
admits at least one solution.\end{prop}

\begin{proof} We use the direct method of the calculus of variations. Let $(T_n)_n$ be a minimizing sequence; since this sequence is bounded in $H^1(\Omega)$, it admits, up to a subsequence, a limit $T$ for the strong convergence in $L^2(\Omega)$; moreover, we also can assume $T_n(x)\to T(x)$ for a.e.~$x\in \Omega$, which also implies that the convergence holds $\mu$-a.e.. This implies that $T$ still satisfies the constraint on the image measure (since, for a continuous and bounded function $\phi$, the a.e.~convergence provides $\int_\Omega \phi\circ T_n \dx\mu \to \int_\Omega \phi\circ T \dx\mu$) and is thus admissible for \eqref{otwithgrad}. On the other hand, it is clear that the functional that we are trying to minimize is lower semi-continuous with respect to the weak convergence in $H^1(\Omega)$. This achieves the proof. \end{proof}

\noindent This existence result can be of course adapted to any functional with form $\int_\Omega L(x,T(x),DT(x)) \dx\mu(x)$ under natural assumptions on the Lagrangian $L$ (continuity with respect to the two first variables, convexity and coercivity with respect to the third one).

\paragraph{Existence of smooth transport maps and Dacorogna-Moser's result.}

In the above existence theorem, the fact that the set of admissible maps is non-empty was an assumption. If the measures are regular enough, it can be seen as a consequence of the classical regularity results of the optimal transport maps for the quadratic cost \cite{caffarelli1,caffarelli2,caffarelli3,dpf}. We recall here another result which provides the existence of a regular diffeomorphism which sends a given measure onto another one, which will be used several times in the paper. This result is due to Dacorogna and Moser \cite{dm} (this construction is nowadays regularly used in optimal transport, starting from the proof by Evans and Gangbo, \cite{eg}; it is also an important tool for the equivalence between different models of transportation, see \cite{dacmosFil}; notice that the transport that they consider is in some sense optimal for a sort of congested transport cost, as pointed out years ago by Brenier in \cite{brenier}). The version of their result that we will use is the following:

\begin{theo} \label{appdacmos} Let $U \subset \R^d$ be a bounded open set with $C^{3,\alpha}$ boundary $\partial U$. Let $f_1$, $f_2$ be two positive Lipschitz functions on $\bar{U}$ such that
$$ \int_U f_1 = \int_U f_2 $$
Then there exists a Lipschitz diffeomorphism $T : \bar{U} \to \bar{U}$ satisfying
$$\left\{
\begin{array}{ll}
\det DT (x) = \dfrac{f_1(x)}{f_2(T(x))}, & x \in U \\[2mm]
T(x) = x, & x \in \partial U
\end{array}
\right. $$
Moreover, the Lipschitz constant of $T$ is bounded by a constant depending only on $U$, on the Lipschitz constants and on the lower bounds of $f_1$ and $f_2$ .
\end{theo}

\noindent Notice that the equation satisfied by $T$ exactly means (as $T$ is Lipschitz and one-to-one) that it sends the measures with density $f_1$ onto the measure with density $f_2$. The result of the original paper \cite[Theorem 1']{dm} deals with different assumptions on the density $f_1$ ($f_1 \in C^{k+3,\alpha}(\bar{U})$ and the result is a $C^{k+1,\alpha}$ diffeomorphism) but only considered $f_2=1$. We gave here a formulation better suitable for our needs, easy to obtain from the original theorem; for the sake of completeness, let us state the following corollary, which will be used several times in practice:

\begin{co} \label{codm} Let $\Omega$, $\Omega'$ be two bounded open sets, and assume that there exists $\Omega''$ open, bounded and with $C^{3,\alpha}$ boundary, and $\psi_1 : \bar\Omega\to\bar{\Omega''}$, $\psi_2:\bar{\Omega'}\to\bar{\Omega''}$ two bi-Lipschitz diffeomorphisms, so that the Jacobian determinants $\det D\psi_1$, $\det D\psi_2$ are also Lipschitz. Let $f_1 \in \Lip(\bar\Omega)$, $f_2 \in \Lip(\bar{\Omega'})$ be two positive functions with $\int_\Omega f_1 = \int_{\Omega'} f_2$. Then there exists a bi-Lipschitz diffeomorphism $T:\bar{\Omega}\mapsto\bar{\Omega'}$ such that:
\begin{itemize}
\item $T$ sends the measure $f_1\cdot\leb^d|_{\Omega}$ to the measure $f_2\cdot\leb^d|_{\Omega'}$;
\item for any $x\in\partial\Omega$, $T(x) = \psi_2^{-1}\circ\psi_1(x)$.
\end{itemize}
Moreover, $\Lip T \leq C$, where $C$ depends only on the lower bounds of $f_1,f_2$, on their Lipschitz constants, on the domains and on $\psi_1$, $\psi_2$. \end{co}

\noindent The proof is elementary, by applying Theorem~\ref{appdacmos} to the domain $\Omega''$ and with the measures $(\psi_1)_\#f_1, (\psi_2)_\#f_2$ (the Lipschitz regularity of the Jacobian determinants is needed to guarantee regularity of these image measures).  We do not claim the assumptions of Corollary~\ref{codm} to be sharp, but they are sufficient for the case we are interested in (see paragraphs~2.2 and~4.2).

\paragraph{The one-dimensional case.} We finish by giving some very partial results about the optimal transport problem with gradient term on the real line. First, we recall the classical result about the one-dimensional optimal transportation problem (we refer for instance to \cite[Chapter~2]{sln} for the proof):

\begin{prop} Let $I$ be a bounded interval of $\R$ and $\mu \in \mathcal{P}(I)$ be atom-less and $\nu \in \mathcal{P}(\R)$. Then there exists a unique map $T:I\to \R$ which is non-decreasing and sends $\mu$ onto $\nu$, and a unique map $U:I\to\R$ which is non-increasing and sends $\mu$ to $\nu$. Moreover, if $h$ is a convex function $\R\to\R$, then the non-decreasing map $T$ solves the minimization problem
$$ \inf\left\{ \int_\Omega h(T(x)-x)\dx\mu(x) \,:\, T_\#\mu = \nu \right\} $$
with uniqueness provided that $h$ is strictly convex. \end{prop}

\noindent Now we state the results concerning the optimality of this monotone map $T$ for the Sobolev norm among the transport maps. 

\begin{prop} Let $I$ be a bounded interval of $\R$. Let $\mu$ be a positive and finite measure on $I$, having a density $f$, and $\nu$ be a positive measure on $\R$ having same mass as $\mu$. Then:
\begin{itemize}
\item Assume $\mu$ is uniform. Then, for any convex and non-decreasing function $h$ on $\R$, both the non-decreasing and the non-increasing maps from $\mu$ to $\nu$ solve the problem
$$ \inf \left\{ \int_I h(|T'(x)|)\dx x \,:\, T_\#\mu = \nu \right\} . $$
\item There exists a constant $\alpha_0 > 1$ so that, if $f$ satisfies $\dfrac{\sup f}{\inf f} \leq \alpha_0$, then either the non-decreasing map or the non-increasing map from $\mu$ to $\nu$ solves
\begin{equation} \inf \left\{ \int_I |T'(x)|^2 \dx x \,:\, T_\#\mu = \nu \right\} .  \label{otderiv1d} \end{equation}
\item On the other hand, for any $\alpha$ large enough, one can find a measure $\mu_\alpha$ having a density $f_\alpha$ satisfying $\dfrac{\sup f_\alpha}{\inf f_\alpha} = \alpha $, and a measure $\nu_\alpha$, for which neither the non-decreasing and the non-increasing transport map from $\mu_\alpha$ to $\nu_\alpha$ are optimal for~\eqref{otderiv1d}.
\end{itemize}
\end{prop}

\noindent The proof of the optimality of the monotone map where $\mu$ is uniform is fully included in~\cite{ls}, and the more general case is a consequence of the results of~\cite[Section~2.1]{lphd}. Concerning the case where $(\sup f)/(\inf f)$ is too large, it is enough to consider the following counter-example: we fix a function $V$ on $[0,1]$, which is equal to the triangle function ({\it i.e.}~$x\mapsto 2x$ on $[0,1/2]$ and $2-2x$ on $[1/2,1]$); and, for $\alpha>0$, we consider a density $f_\alpha$ which takes the value~$1$ on $[1/4,3/4]$, and the value $\alpha$ elsewhere on $[0,1]$. It is then easy to compute the unique non-decreasing map from $\mu_\alpha$ to the image measure $\nu_\alpha := V_\# \mu_\alpha$; denoting by $T_\alpha$ this map, we then check that $\int |V'|^2 < \int |T_\alpha'|^2$ for $\alpha$ large enough, and the same holds for the non-increasing map $U_\alpha$. The same construction has been used in order to build similar counter-examples in other kind of ``regularized'' transport problems (see~\cite{lps}).


\subsection{Definitions and basic results of $\Gamma$-convergence}


We finish this preliminary section by the tools of $\Gamma$-convergence that we will use throughout this paper. All the details can be found, for instance, in the classical Braides's book \cite{braides}. In what follows, $(X,d)$ is a metric space.

\begin{defi} Let $(F_n)_n$ be a sequence of functions $X \mapsto \bar\R$. We say that $(F_n)_n$ $\Gamma$-converges to $F$, and we write $F_n \xrightarrow[n]{\Gamma} F$ if, for any $x \in X$, we have
\begin{itemize}
\item for any sequence $(x_n)_n$ of $X$ converging to $x$,
$$ \liminf\limits_n F_n(x_n) \geq F(x) \qquad \text{($\Gamma$-liminf inequality);}$$
\item there exists a sequence $(x_n)_n$ converging to $x$ and such that
$$ \limsup\limits_n F_n(x_n) \leq F(x) \qquad \text{($\Gamma$-limsup inequality).} $$
\end{itemize} \end{defi}

This definition is actually equivalent to the following equalities for any $x \in X$:
$$ F(x) = \inf\left\{ \liminf\limits_n F_n(x_n) : x_n \to x \right\} = \inf\left\{ \limsup\limits_n F_n(x_n) : x_n \to x \right\} $$
The function $x \mapsto  \inf\left\{ \liminf\limits_n F_n(x_n) : x_n \to x \right\}$ is called $\Gamma$-liminf of the sequence $(F_n)_n$, and the other one its $\Gamma$-limsup. A useful result is the following (which, for instance, implies that a constant sequence of functions does not $\Gamma$-converge to itself in general):

\begin{prop} The $\Gamma$-liminf and the $\Gamma$-limsup of a sequence of functions $(F_n)_n$ are both lower semi-continuous on $X$. \end{prop}



\noindent The main interest of $\Gamma$-convergence is its consequences in terms of convergence of minima:

\begin{theo} \label{convminima} Let $(F_n)_n$ be a sequence of functions $X \to \bar\R$ and assume that $F_n \xrightarrow[n]{\Gamma} F$. Assume moreover that there exists a compact and non-empty subset $K$ of $X$ such that
$$ \forall n\in N, \; \inf_X F_n = \inf_K F_n $$
(we say that $(F_n)_n$ is equi-mildly coercive on $X$). Then $F$ admits a minimum on $X$ and the sequence $(\inf_X F_n)_n$ converges to $\min F$. Moreover, if $(x_n)_n$ is a sequence of $X$ such that
$$ \lim_n F_n(x_n) = \lim_n (\inf_X F_n)  $$
and if $(x_{\phi(n)})_n$ is a subsequence of $(x_n)_n$ having a limit $x$, then $ F(x) = \inf_X F $. \end{theo}

\noindent We finish with the following result, which allows to focus on the $\Gamma$-limsup inequality only on a dense subset of $X$ under some assumptions:

\begin{prop} \label{gammadense} Let $(F_n)_n$ be a sequence of functionals and $F$ be a functional $X \to \bar\R$. Assume that there exists a dense subset $Y \subset X$ such that:
\begin{itemize}
\item for any $x \in X$, there exists a sequence $(x_n)_n$ of $Y$ such that $x_n \to x$ and $F(x_n) \to F(x)$;
\item the $\Gamma$-limsup inequality holds for any $x \in Y$.
\end{itemize}
Then it holds for any $x$ belonging to the whole $X$. \end{prop}

\section{Generalities and density of Sobolev transport maps}

In what follows, we consider two regular enough domains $\Omega$, $\Omega'$ and two measures $\mu \in \mathcal{P}(\Omega)$, $\nu \in \mathcal{P}(\Omega')$ with positive and bounded from below densities $f$, $g$; we assume moreover that the class of maps $T \in H^1(\Omega)$ sending $\mu$ onto $\nu$ is non-empty (for instance, the assumptions of the Dacorogna-Moser's result are enough). The functional that we will study is defined, for $\eps > 0$, by
$$ J_\eps \,:\, T \mapsto \int_\Omega |T(x)-x|\dx\mu(x) + \eps \int_\Omega |DT(x)|^2 \dx x $$
and we denote by $J$ the corresponding functional when $\eps =0$ (which is, thus, the classical Monge's transport energy); moreover, we extend $J_\eps$, $J$ to the whole $L^2(\Omega)$ by setting $J_\eps(T)=J(T)=+\infty$ for a map $T$ which is not a transport map from $\mu$ to $\nu$.

As usual in transport theory, we consider as a setting for our variational problems  the set of transport plans $\gamma$ which are probabilities on the product space $\Omega\times\Omega'$ with given marginals $(\pi_x)_\#\gamma=\mu$ and $(\pi_y)_\#\gamma=\nu$ and all the $\Gamma$-limits that we consider in that follows are considered with respect to the weak convergence of plans as probability measures. However, due to our choices of the functionals that we minimize, most of the transport plan that we consider will be actually induced by transport maps, {\it i.e.}~$\gamma_T = (\id\times T)_\#\mu$ with $T_\#\mu=\nu$. These maps are valued in $\Omega'$, which is bounded, and are hence bounded. We could also consider different notions of convergence, in particular based on the pointwise convergence of these plans, and we will actually do it often. For simplicity, we will use the convergence in $L^2(\Omega;\Omega')$ (but, since these functions are bounded, this is equivalent to any other $L^p$ convergence with $p<\infty$).  As the following lemma (which will be also technically useful later) shows, this convergence is equivalent to the weak convergence in the sense of measures of the transport plans:

\begin{lem} \label{plansandmaps} Assume that $\Omega$, $\Omega'$ are compact domains and $\mu$ is a finite non-negative measure on $\Omega$. Let $(T_n)_n$ be a sequence of maps $\Omega \to \Omega'$. Assume that there exists a map $T$ such that $\gamma_{T_n} \rightharpoonup \gamma_T$ in the weak sense of measures. Then $T_n \to T$ in $L^2_\mu(\Omega)$.

Conversely, if $T_n \to T$ in $L^2_\mu(\Omega)$, then we have $\gamma_{T_{n}} \rightharpoonup \gamma_T$ 
in the weak sense of measures. \end{lem}

\begin{proof} If $\gamma_{T_n} \rightharpoonup \gamma_T$ and $\phi \in C_b(\Omega)$ is a vector-valued function, we have
$$\int_\Omega \phi(x) \cdot T_n(x) \dx\mu(x) = \int_{\Omega\times \Omega'} \phi(x) \cdot y \dx\gamma_{T_n}(x,y) \to \int_{\Omega\times \Omega'}\phi(x)\cdot y \dx\gamma_T(x,y) = \int_{\Omega} \phi(x) \cdot T(x) \dx\mu(x) $$
which proves that $T_n \rightharpoonup T$ weakly in $L^2_\mu(\Omega)$. On the other hand,
$$ \int_\Omega |T_n(x)|^2 \dx\mu(x) = \int_{\Omega\times\Omega'} |y|^2 \dx\gamma_{T_n}(x,y) \to \int_{\Omega\times\Omega'} |y|^2\dx\gamma_T(x,y) = \int_\Omega |T(x)|^2 \dx\mu(x) $$
$$ \text{thus} \qquad ||T_n||_{L^2_\mu} \to ||T||_{L^2_\mu} $$
and the convergence $T_n \to T$ is actually strong.

Conversely, assume that $T_n \to T$ in $L^2_\mu(\Omega)$ and let $(n_k)_k$ be such that the convergence $T_{n_k}(x) \to T(x)$ holds for a.e.\@~$x\in\Omega$, then for any $\phi \in C_b(\Omega\times\Omega')$ we have
$$ \int_{\Omega\times\Omega'} \phi(x,y) \dx\gamma_{T_{n_k}}(x,y) = \int_\Omega \phi(x,T_{n_k}(x)) \dx\mu(x) \to \int_\Omega \phi(x,T(x))\dx\mu(x) = \int_{\Omega\times\Omega'} \phi(x,y) \dx\gamma_T(x,y). $$
This proves $\gamma_{T_{n_k}} \rightharpoonup \gamma_T$, but, the limit being independent of the subsequence, we easily get the convergence of the whole sequence.
\end{proof}

\noindent Since the set of transport plans between $\mu$ to $\nu$ is compact for the weak topology in the set of measures on $\Omega\times\Omega'$, a consequence of Lemma \ref{plansandmaps} is that the equi-coercivity needed in Theorem~\ref{convminima} will be satisfied in all the $\Gamma$-convergence results that follows. Therefore, we will not focus on it anymore and still will consider that these results imply the convergence of minima and of minimizers.

\subsection{Statements of the zeroth and first order $\Gamma$-convergences}

\paragraph{Zeroth order $\Gamma$-limit.} The first step consists in checking that $J_\eps \gammato J$. Here we must consider that $J_\eps$ is extended to transport plan by setting $+\infty$ on those transport plans which are not of the form $\gamma=\gamma_T$ for $T\in H^1$, and that $J$ is defined as usual as $J(\gamma)=\int |x-y|\dx\gamma$ for transport plans. The proof of the $\Gamma$-convergence uses Theorem~\ref{denssobolev} on density of Sobolev transports, that we prove below in Paragraph~\ref{proofdensity}. This density result holds for H\"older and bounded from below densities, and for a large class of domains that we define as follows:

\begin{defi} \label{lipdom} We call {\em Lipschitz polar domain} any open bounded subset $\Omega$ of $\R^d$ having form
$$ \Omega = \left\{ x \in \R^d \,:\, |x-x_0| < \gamma\left( \frac{x-x_0}{|x-x_0|}\right) \right\} $$
for some $x_0 \in \Omega$ and a Lipschitz function $\gamma:S^{d-1}\to (0,+\infty)$. In particular, such a domain $\Omega$ is star-shaped with Lipschitz boundary. \end{defi}

\begin{prop}[Zeroth order $\Gamma$-limit] \label{zerogammalim} Assume that $\Omega$, $\Omega'$ are both Lipschitz polar domains and that $f$, $g$ are both $C^{0,\alpha}$ and bounded from below. Then $J_\eps \gammato J$ as $\eps \to 0$. \end{prop}

\begin{proof} The $\Gamma$-liminf inequality is trivial (we have $J_\eps \geq J$ by definition, and $J$ is continuous for the weak convergence of plans), and the $\Gamma$-limsup inequality is a direct consequence of the Prop.\@ \ref{gammadense} and of the density of the set of Sobolev transports for the $L^2$-convergence. \end{proof}

\paragraph{First order $\Gamma$-limit.} We state it as follows, with this time a short proof:

\begin{prop}[First order $\Gamma$-limit] \label{gammalim} Assume simply that $f$, $g$ are bounded from below on $\Omega$, $\Omega'$ and that $\Omega$ has Lipschitz boundary. Then the functional $\dfrac{J_\eps-W_1}{\eps}$ $\Gamma$-converges, when $\eps\to 0$, to
$$ \mathcal H : T \mapsto\begin{cases}\int_\Omega |DT(x)|^2 \dx x & \text{if } T \in \O_1(\mu,\nu)\cap H^1(\Omega) \\ +\infty & \text{otherwise}, \end{cases}$$ 
where, again, $\mathcal H$ is extend to plans which are not induced by maps by $+\infty$. \end{prop}

\begin{proof}
{\it $\Gamma$-limsup inequality.} 
If $T \in \O_1(\mu,\nu)$, then by choosing $T_\eps = T$ for any $\eps$ we obtain automatically $\dfrac{J_\eps(T_\eps)-W_1}{\eps} = \int_\Omega |DT|^2$ for each $\eps$. It remains to show that if $T \notin \O_1(\mu,\nu)$, then we have $\dfrac{J_\eps(T_\eps)-W_1}{\eps} \to +\infty$ for any sequence $(T_\eps)_\eps$ converging to $T$; but since the map
$$ T \mapsto \int_\Omega |T(x)-x| \dx\mu(x) $$
is continuous for the $L^2$-convergence, we have for such a $(T_\eps)_\eps$
$$ \liminf\limits_\eps \frac{J_\eps(T_\eps)-W_1}{\eps} \geq \liminf\limits_\eps\frac{1}{\eps} \left(\int_\Omega |T(x)-x| \dx\mu(x) - W_1\right) $$
which is $+\infty$ since $T$ is not optimal for the Monge problem. 

 {\it $\Gamma$-liminf inequality.} We can concentrate on sequence of maps $T_\eps$ with equibounded values for $\dfrac{J_\eps(T_\eps)-W_1}{\eps}$, which provides a bound on $\int_\Omega |DT_{\eps}|^2 $. 
Assuming the liminf to be finite, from 
$$C\geq  \frac{J_{\eps}(T_{\eps})-W_1}{\eps} = \frac{1}{\eps} \left(\int_\Omega |T_{\eps}(x)-x|\dx\mu(x)-W_1\right) + \int_\Omega |DT_{\eps}|^2 \geq \int_\Omega |DT_{\eps}|^2 $$
we deduce as above that $T$ must belong to $ \O_1(\mu,\nu)$ and, since the last term is lower semi-continuous with respect to the weak convergence in $H^1(\Omega)$ (which is guaranteed up to subsequences since $(J_{\eps_k}(T_{\eps_k}))_k$ is bounded) we get the inequality we look for.\end{proof}

\noindent Since the $\Gamma$-convergence implies the convergence of minima, we then have
$$ \inf J_\eps = \inf J + \eps \inf \mathcal H + o(\eps) = W_1 + \eps \inf \mathcal H + o(\eps)$$
provided that the infimum is finite, which means that there exists at least one transport optimal map~$T$ also belonging to the Sobolev space $H^1(\Omega)$. In the converse case, and under the assumptions of the zeroth order $\Gamma$-convergence, we have
$$ \inf J_\eps \to W_1 \qquad \text{and} \qquad \frac{\inf J_\eps-W_1}{\eps} \to +\infty $$
which means that the lowest order of convergence of $\inf J_\eps$ to $J$ is smaller as $\eps$. The study of a precise example where this order is $\eps|\log\eps|$ is the object of Section \ref{sectionepslogeps}. \bigskip

\paragraph{What about the selected map ?} The first-order $\Gamma$-convergence and the basic properties of $\Gamma$-limits imply that, if $T_\eps$ minimizes $J_\eps$, then $T_\eps \to T$ which minimizes the Sobolev norm among the set $\O_1(\mu,\nu)$ of optimal transport maps from $\mu$ to $\nu$. This gives a selection principle, {\it via} a secondary variational problem (minimizing something in the class of minimizers), in the same spirit of what we presented for the monotone transport map along each transport ray. A natural question is to find which is this new ``special'' selected map, and whether it can coincide with the monotone one. Thanks to the non-optimality results of this map for the Sobolev cost on the real line, the answer is that they are in general  different. We can look at the following explicit counter-example (where we have however $\O_1(\mu,\nu)\cap H^1(\Omega) \neq \emptyset$).

 Let us set $ \Omega = (0,1)^2$, $\Omega' = (2,3)\times(0,1)$ in $\R^2$. Let $F$, $G$ be two probability densities on the real line, supported in $(0,1)$ and $(2,3)$ respectively; we now consider the densities defined by
$$ f(x_1,x_2)=F(x_1) \qquad \text{and} \qquad g(x_1,x_2)=G(x_1)$$
Then, if $t$ is a transport map from $F$ to $G$ on the real line and $T(x_1,x_2) = (t(x_1),x_2)$, it is easy to check that $T$ sends the density $f$ onto $g$ and if $u(x_1,x_2) = x_1$ we have
$$ |T(x)-x| = |t(x_1)-x_1| = t(x_1)-x_1 = u(T(x_1,x_2))-u(x_1,x_2) $$
This proves that $T$ is optimal and $u$, which is of course 1-Lipschitz, is a Kantorovich potential, so that the maximal transport rays are exactly the segments $[0,3]\times\{x_2\}$, $0 < x_2 < 1$. As a consequence,
$$ T \in \O_1(\mu,\nu) \qquad \iff \qquad T(x_1,x_2) = (t(x_1,x_2),x_2) \quad \text{with} \quad \begin{array}{c} t(\cdot,x_2)_\# F = G \\ \text{for a.e.}~x_2 \end{array} $$
In particular, the monotone transport map along the maximal transport rays is $x \mapsto (t(x_1),x_2)$, where $t$ is the non-decreasing transport map from $F$ to $G$ on the real line. For this transport map $T$, we have
$$ \int_\Omega |DT(x)|^2 \dx x = \int_0^1 t'(x_1)^2 \dx x_1 $$
Now, one can choose $F$, $G$ such that the solution of
\begin{equation} \inf\left\{ \int_0^1 U'(x_1)^2 \dx x_1 \;:\; U_\#F = G \right\} \label{t1d} \end{equation}
is not attained by the increasing transport map from $F$, to $G$. Thus, if $\tilde{t}$ minimizes \eqref{t1d} and $\tilde{T}(x_1,x_2) = \tilde{t}(x_1)$, we have
$$ \int_\Omega |D\tilde{T}(x)|^2 \dx x = \int_0^1 \tilde{t}'(x_1)^2 \dx x_1 < \int_0^1 t'(x_1)^2 \dx x_1 = \int_\Omega |DT(x)|^2 \dx x $$
and $\tilde{T}$ is also an optimal transport map for the Monge problem. 

\subsection{Proof of the density of Lipschitz transports \label{proofdensity}}

In this section, we focus on the proof of the density of the set of Lipschitz transport maps in the set of all the transport maps. Before starting the proof of this result, let us give some quick comments on it.

First of all, for the sake of the applications to the zero-th order $\Gamma$-convergence of the previous section, we only need the density of those maps belonging to the Sobolev space $H^1(\Omega)$; also, we needed density in the set of plans, but since it is well known that transport maps are dense in the set of transport plans, for simplicity we will prove density in the set of maps. Concerning the assumptions on the measure to be H\"older and bounded from above and below, we notice that they are also used in several classical results on regularity of transport maps ({\it cf.}~the Caffarelli's regularity theory, or even the Dacorogna and Moser's result above).

Finally, the assumptions on the domains are only needed for technical reasons; basically, they allow to send them onto the unit ball through a Lipschitz diffeomorphism. This is the aim of the following lemma, that we will use several times:

\begin{lem} \label{lemdiffeo} If $U$ is a Lipschitz polar domain with
$$ U = \left\{ x \in \R^d \,:\, |x-x_0| < \gamma\left( \frac{x-x_0}{|x-x_0|}\right) \right\} $$
then there exists a map $\alpha:U\to B(0,1)$ such that:
\begin{itemize}
\item $\alpha$ is a bi-Lipschitz diffeomorphism from $U$ to $B(0,1)$;
\item $\det D\alpha$ is Lipschitz and bounded from below (thus, $\det D\alpha^{-1}$ is also Lipschitz);
\item for any $x \neq x_0$, $\dfrac{\alpha(x)}{|\alpha(x)|} = \dfrac{x-x_0}{|x-x_0|}$
\end{itemize} \end{lem}

\noindent (We will not use the third property in the proof of Theorem \ref{denssobolev}, but it will be useful later in section~\ref{construction}).

\begin{proof} Up to a translation and a dilation, we can assume $x_0=0$ and $\gamma \leq 1$ on $S^{d-1}$. Now we set
$$ \alpha(x) = \begin{cases} x \qquad \text{if } |x| \leq \dfrac 12 \gamma\left(\dfrac{x}{|x|}\right) \\ \lambda(x)\dfrac{x}{|x|} \qquad \text{otherwise} \end{cases}. $$
for a suitable choice of the function $\lambda: U \to [0,+\infty)$. Assuming that $\lambda$ is Lipschitz, we compute $ D\alpha$ on the region $\left\{|x|> \dfrac 12 \gamma(x/|x|) \right\}$. Here we have $D\alpha = x \otimes \nabla\left(\dfrac{\lambda(x)}{x}\right) + \dfrac{\lambda}{|x|} I_d $ thus, in an orthonormal basis whose first vector is $e = \dfrac{x}{|x|}$,
$$ D\alpha = \left(\begin{array}{llll}
|x| \partial_e \left(\dfrac{\lambda(x)}{|x|}\right) & |x| \partial_{e_2} \left(\dfrac{\lambda(x)}{|x|}\right) & \dots & |x|\left(\partial_{e_n} \dfrac{\lambda(x)}{|x|}\right) \\
0 & 0 & \dots 0 \\
\vdots & \vdots & \dots & \vdots \\
0 & 0 & \dots 0 \\
\end{array} \right) + \frac{\lambda}{|x|} I_n
 $$
$$ \text{and} \qquad \det D\alpha = \left(|x|\partial_e \left(\dfrac{\lambda(x)}{|x|}\right)+\frac{\lambda(x)}{|x|}\right) \left(\frac{\lambda(x)}{|x|}\right)^{n-1}  $$
We write $\lambda = \lambda(r,e)$ where $r = |x|$, which leads to
$$ \det D\alpha = \left(r \partial_r\left(\frac{\lambda}{r}\right)+\frac{\lambda}{r}\right)\left(\frac{\lambda}{r}\right)^{n-1} = \partial_r\lambda \left(\frac{\lambda}{r}\right)^{n-1}$$
At this time, we see that the following conditions on $\lambda$ allow to conclude:
\begin{itemize}
\item $\lambda$ is Lipschitz on the domain $\left\{ \frac{\gamma(x/|x|)}{2} \leq |x| \leq \gamma(x/|x|) \right\}$;
\item if we fix $e \in S^{d-1}$, then $\lambda(\cdot,e)$ is increasing on the interval $\left[\frac{\gamma(e)}{2},\gamma(e)\right]$ ;
\item $\lambda\left(\frac{\gamma(e)}{2},e\right) = \frac{\gamma(e)}{2}$ and $\lambda(\gamma(e),e) = 1$;
\item $\partial_r\lambda\left(\frac{\lambda }{r}\right)^{n-1}$ is Lipschitz and is equal to $1$ for $r=\frac{\gamma(e)}{2}$, which means $\partial_r\lambda\left(\frac{\gamma(e)}{2},e\right) = 2^{n-1}$.
\end{itemize}
To satisfy these conditions, it is enough to choose for $\lambda(\cdot,e)$ a second-degree polynomial function with prescribed values at $\frac{\gamma(e)}{2}$, $\gamma(e)$ and prescribed first derivative at $\frac{\gamma(e)}{2}$.
\end{proof}

\begin{proof}[Proof of Theorem \ref{denssobolev}] In a first time, we assume $\Omega$ and $\Omega'$ to be both equal to the unit square $\left[-\frac{1}{2},\frac{1}{2}\right]^d$; we will generalize later to any pair of Lipschitz polar domains (see Step~4 below). Moreover, we will also identify the cube~$\Omega$ with a Torus, in order to define regularizations by convolution on it. Let then $T:\Omega\to\Omega$ be a transport from $\mu$ to $\nu$; our goal is to build a sequence $(T_n)_n$ of Lipschitz maps such~that
$$ T_n \xrightarrow[n\to+\infty]{L^2(\Omega)} T \qquad \text{and} \qquad \forall n  \in \N, \; (T_n)_\#\mu = \nu $$

{\it \underline{Step 1:} regularization of the transport plan.}

We denote by $\gamma$ the transport plan associated to $T$; let us recall that it is defined by
$$ \iint_{\Omega^2} \phi(x,y) \dx\gamma(x,y) = \int_{\Omega} \phi(x,T(x)) f(x) \dx x $$
for any continuous and bounded function $\phi$ on $\Omega\times\Omega'$.


We then consider a standard sequence of convolution kernels $(\rho_k)_k$, approximating the identity, and define a sequence of measures $(\gamma_k)_k$ by
$$\dx\gamma_k(x,y) = \rho_k(y-T(x)) f(x) \dx y \dx x$$
for each $k$ (here $y-T(x)$ has to be considered in a periodic sense, and this is why we identified $\Omega$ with the torus; in this way, there is no loss of mass in this approximation). By construction, $\gamma_k$ has a density on $\Omega^2$ which is smooth and bounded from below by a positive constant $c_k$. We also notice that:
\begin{itemize}
\item $\gamma_k$ has $\mu$ as first marginal, for any $k$; moreover, its second marginal, that we will denote by $\nu_k$, has also a smooth density which is bounded from below by a positive constant;
\item we have the convergence $\gamma_k \wto \gamma$ in the weak sense of measures; as a consequence, the convergence $\nu_k \wto \nu$ holds as well.
\end{itemize}



{\it \underline{Step 2:} construction of a transport map corresponding to the regularized transport plan.} 

The goal of this step is to build, for each $k$, a family of regular maps $(T^k_h)_h$ sending $\mu$ to $\nu_k$ and so that, up to a suitable diagonal extraction, the convergence $T^k_{h_k} \to T$ holds in $L^2(\Omega)$. We fix $k \in \N$ and a dyadic number $h$, and decompose $\Omega = \bigcup\limits_i Q_i$ into a union of cubes $Q_i$, each with size-length~$2h$, via a regular grid. For each $i$, we set
$$ \mu_i = \mu|_{Q_i} = (\pi_1)_\#(\gamma_k|_{(Q_i\times\Omega)}) \qquad \text{and} \qquad \nu_i^k = (\pi_2)_\#(\gamma_k|_{(Q_i\times\Omega)}) $$
It is clear that these two measures have same mass, the first one having $Q_i$ as support and the second one~$\Omega$. Moreover, the assumptions on~$\mu$ and the remarks above about $\gamma_k$ guarantee that $\mu_i$ and $\nu_i$ have both a density which is Lipschitz and bounded from below. The end of this step consists in building a map $U^k_{i,h}$, from the cube $Q_i$ to $\Omega$, such that:
\begin{itemize}
\item $U_{i,h}^k$ is Lipschitz continuous;
\item $U_{i,h}^k(x) = x$ for $x \in \partial Q_i$ (notice that, although the source and the target set of $U_{i,h}^k$ are different, this makes sense since $Q_i$ is included in $\Omega$);
\item $U_{i,h}^k$ sends the measure $\mu_i$ onto $\nu_i^k$.
\end{itemize}
We then define a global map $T_h^k : \Omega \to \Omega$ by $T_h^k(x) = U_{i,h}^k(x)$ on each $Q_i$. We notice that the first point above guarantees that $T_h^k$ is Lipschitz on each $Q_i$, the second one that it is continuous on each common boundary $Q_i\cap Q_j$ (thus it is globally Lipschitz on $\Omega$), and the third one that it sends $\mu$ to $\nu_k$, as required.

Let us now explain very briefly the construction of such a $U_{i,h}^k$. For the sake of simplicity, we give the details only in the case where $Q_i$ is not an extreme cube of $\Omega$, {\it i.e.}~where $\partial Q_i \cap \partial \Omega$ is empty, and claim that the extreme case can be treated in a similar way.

We denote by $x_i$ the center of $Q_i$ which has radius $h$ (for the infinite norm in $\R^d$), and by $\delta_1$ and $\delta_2$ the two positive numbers such that
$$ \mu( Q_i \setminus B_\infty(x_i,\delta_1)) = \frac{1}{2} \nu_i^k(\Omega\setminus Q_i) $$
$$ \mu(B_\infty(x_1,\delta_1)\setminus B_\infty(x_i,\delta_2)) = \frac{1}{2} \nu_i^k(\Omega\setminus Q_i)  $$
$$\text{and} \qquad \mu(B_\infty(x_i,\delta_2)) = \nu_i^k(Q_i) $$
where $B_\infty$ denotes the unit ball for the norm $||\cdot||_\infty$. Then:
\begin{itemize}
\item with an identical proof as Lemma~\ref{lemdiffeo}, we may find two Lipschitz diffeomorphisms $\psi_1^1$ and $\psi_1^2$, with Lipschitz Jacobian determinants, sending $\bar{Q_i} \setminus B_\infty(x_i,\delta_1)$ and $\bar{\Omega}\setminus Q_i$ to $\bar{B}(0,1)\setminus B(0,1/2)$ (where $B$ denotes this time the unit ball for the Euclidean norm) and having the property $(\psi_2^1)^{-1}\circ\psi_1^1(x) = x$ on the boundary $\partial Q_i$. Applying our version of the Dacorogna-Moser's result (Corollary~\ref{codm}) gives us a fist Lipschitz map~$V_1$ so that
$$ (V_1)|_\# (\mu|_{Q_i\setminus B_\infty(x_i,\delta_1)}) = \frac{1}{2} \nu_i^k|_{\Omega\setminus Q_i} \qquad \text{and} \qquad V_1(x)=x \; \text{on}\; \partial Q_i $$
\item There exists also two diffeomorphisms $\psi_1^2 : \bar{B}_\infty (x_i,\delta_1)\setminus B_\infty(x_i,\delta_2)$ and $\psi_2^2 : \bar\Omega\setminus Q_i \to \bar B(0,1)\setminus B(0,1/2)$, with the property $(\psi_2^2)^{-1}\circ\psi_1^2(x) = V_1(x)$ on $\partial B_\infty (x_i,\delta_1)$;  again, Corollary~\ref{codm} provides a Lipschitz map $V_2$ such that
$$ (V_2)|_\# (\mu|_{B_\infty(x_i,\delta_1) \setminus B_\infty(x_i,\delta_2)}) = \frac{1}{2} \nu_i^k|_{\Omega\setminus Q_i} \qquad \text{and} \qquad V_2(x)=V_1(x) \; \text{on}\; \partial B_\infty(x_i,\delta_1)  $$
\item Similarly, we can find a Lipschitz map $V_3$ such that
$$ (V_3)|_\#(\mu|_{B_\infty(x_i,\delta_2)}) = \nu_i^k|_{Q_i} \qquad \text{and} \qquad V_3(x) = V_2(x)\; \text{on}\; \partial B_\infty (x_i,\delta_2) $$
\end{itemize}
The maps $V_1$, $V_2$, $V_3$ being Lipschitz and coinciding on the interfaces, the global map defined by
$$ U_{i,h}^k = \left\{\begin{array}{ll} V_1 & \text{on } Q_i\setminus B_\infty(x_i,\delta_1)\\  V_2 & \text{on } B_\infty(x_i,\delta_1) \setminus B_\infty(x_i,\delta_2) \\ V_3 & \text{on } B_\infty(x_i,\delta_2) \end{array} \right. $$
is itself Lipschitz; moreover, the properties of the image measures of $V_1$, $V_2$, $V_3$ imply that $U_{i,h}^k$ sends, as required, the measure $\mu_i$ onto $\nu_i^k$

It then remains to check that, with this construction, $T_{h_k}^k \to T$ in $L^2(\Omega)$ for a suitable sequence $(h_k)_k$. We will prove that, in fact, we have $W_1(\gamma_{T_{h_k}^k},\gamma_T) \to 0$, where $W_1$ is the Wasserstein distance on the set $\mathcal{P}(\Omega^2)$ and $\gamma_{T_{h,k}^k}$ is the transport plan associated to the map $T_{h,k}^k$. Since the Wasserstein distance metrizes the convergence of transport plans, and thanks to Lemma~\ref{plansandmaps} above, we will conclude that the $L^2$-convergence that we want holds.

Let $u$ be a $1$-Lipschitz function $\Omega^2 \to \R$. Let us compute
$$ \int_{\Omega^2} u(x,y) \dx\gamma_{T_h^k}(x,y) - \int_{\Omega^2} u(x,y) \dx\gamma_k(x,y) = \sum\limits_i \int_{Q_i\times\Omega} u(x,y) (\dx\gamma_{T_h^k}-\dx\gamma_k)(x,y) $$
Since $\gamma_k$ and $\gamma_{T_h^k}$ have both the same second marginal on each $Q_i\times \Omega$,
$$ \int_{Q_i\times\Omega} u(x_i,y) (\dx\gamma_{T_h^k}(x,y)-\dx\gamma_k(x,y)) = 0$$
where, for each $i$, $x_i$ denotes the center of $Q_i$. Thus 
\begin{multline*} \int_{\Omega^2} u(x,y) \dx\gamma_{T_h^k}(x,y) - \int_{\Omega^2} u(x,y) \dx\gamma_k(x,y) = \sum\limits_i \int_{Q_i\times\Omega} (u(x,y)-u(x_i,y)) (\dx\gamma_{T_h^k}-\dx\gamma_k)(x,y)
\\\leq \sum\limits_i \int_{Q_i\times\Omega} |x-x_i| (\dx\gamma_{T_h^k}-\dx\gamma_k)(x,y) \end{multline*}
where in the last inequality we used the fact that $u$ is supposed to be $1$-Lipschitz. We deduce
$$  \int_{\Omega^2} u(x,y) \dx\gamma_{T_h^k}(x,y) - \int_{\Omega^2} u(x,y) \dx\gamma_k(x,y) \leq h \iint_{\Omega^2} (\dx\gamma_{T_h^k}-\dx\gamma_k) \leq 2h $$
This inequality holds for any $1$-Lipschitz function $u:\Omega^2\to\R$. Thanks to the duality formulation of the Monge problem, we deduce that $k$, $W_1(\gamma_k,\gamma_{T_h^k}) \leq h$ for any $k$. Since we know that $\gamma_k \wto \gamma_T$ as $k\to+\infty$, we also have $W_1(\gamma_k,\gamma_T) \to 0$. It is then enough to set $h_k = 1/{2^k}$, so that $W_1(\gamma_{T_{h_k}^k},\gamma_T) \to 0$ as required.

{\it \underline{Step 3:} rearranging $T_h^k$.} 

The goal of this step is to compose each map $T_h^k$ with a map $U_k$ sending $\nu_k$ to $\nu$; the obtained map $T_k := U_k\circ T_h^k$ will then send $\mu$ to $\nu$, and we must also chose $U_k$ so that $T_k \to T$. We use the classical regularity result of Caffarelli (see \cite[Theorem~4.4]{vil}): if 
\begin{itemize}
\item $X$, $Y$ are two bounded open sets of $\R^d$, uniformly convex and with $C^2$ boundary;
\item $f \in C^{0,\alpha}(\bar{X})$, $g \in C^{0,\alpha}(\bar{Y})$ are two probability densities
\end{itemize}
then the optimal transport map from $f$ to $g$ belongs to $C^{1,\alpha}(\bar{X})$. The assumption on the source domain is of course not satisfied here, thus we use Lemma~\ref{lemdiffeo} to get a bi-Lipschitz map $\alpha:\Omega\to B(0,1)$ such that $\det D\alpha$ is also Lipschitz. We then denote by
$$ \bar{g} = \alpha_\# g \qquad \text{and} \qquad \bar{g}_k = \alpha_\# g_k $$
where $g_k$ is the density of $\nu_k$ (we know that $g_k$ is Lipschitz and bounded from below). The fact that $\det D\alpha$ is Lipschitz and bounded guarantees that $\bar{g}$ is H\"older and $\bar{g}_k$ is Lipschitz.

We know that $g_k \rightharpoonup g$ in the weak sense of measures, thus $\bar{g}_k \rightharpoonup \bar{g}$ as well, and
$$ W_2(\bar{g}_k,\bar{g}) = \int_{B(0,1)} |U_k(x)-x|^2 \bar{g}_k(x) \dx x \to 0  $$
where $U_k$ is the optimal transport map from $\bar{g}_k$ to $\bar{g}$ for the quadratic cost; the Caffarelli's regularity result guarantees $U_k$ to belong to $C^{1,\alpha}(\bar{B}(0,1))$, and we deduce from the convergence $W_2(\bar{g}_k,{g}) \to 0$ that $U_k(x) \to x$ for almost any $x \in B(0,1)$. Now if we consider
$$ \tilde{T}_k = \alpha^{-1}\circ U_k \circ \alpha \circ T_h^k  $$
then we can check that $\tilde{T}_k \to T$ in $L^2(\Omega)$, and we have $(\tilde{T}_k)_\#\mu = \nu$ for each $k$ by construction. The proof is complete in the case $\Omega=\Omega'=\left[-\frac{1}{2},\frac{1}{2}\right]^d$.

{\it \underline{Step 4:} generalization to any pair of domains.} 

The result being proven in the particular case of the unit cube, the last step consists in generalizing it to a generic pair $(\Omega,\Omega')$ of Lipschitz polar domains. Given such a pair and a transport map $T$ from $\mu$ to $\nu$, we consider two Lipschitz diffeomorphisms $\alpha_1:\Omega\to[0,1]^d$, $\alpha_2:\Omega'\to[0,1]^d$ as in the Lemma \ref{lemdiffeo}. The regularity of $\det D\alpha_1$, $\det D\alpha_2$ guarantees that the image measures $(\alpha_1)_\#\mu$, $(\alpha_2)_\#\nu$ have both $C^{0,\alpha}$ regularity, thus we are able to find a sequence $(U_k)_k$ of transport maps from $(\alpha_1)_\#\mu$ to $(\alpha_2)_\#\nu$ converging for the $L^2([0,1]^d)$-norm to $(\alpha_2)\circ T \circ (\alpha_1)^{-1}$, and it is now easy to check that $T_k = (\alpha_2)^{-1}\circ U_k\circ\alpha_1$ sends $\mu$ to $\nu$ and converges to $T$. \end{proof}

\section{An example of approximation of order $\eps|\log\eps|$ \label{sectionepslogeps}}

As we said in the above section, the convergence $\frac{J_\eps-W_1}{\eps} \to \mathcal H$ allows to know the behavior of $\inf J_\eps$ and of any family $(T_\eps)_\eps$ of minimizers in the case where $\inf\mathcal H<+\infty$. This section is devoted to the study of an example where this assumption fails ({\it i.e.}~where any optimal transport map for the Monge problem is not Sobolev).

\subsection{Notations and structure of the optimal maps}

In the rest of this paper, we set $d = 2$ and we will denote by $(r,\theta)$ the usual system of polar coordinates in $\R^2$. Our source and target domains will be respectively
$$ \Omega = \left\{ x = (r,\theta) : 0 < r < 1, 0 < \theta < \frac{\pi}{2} \right\} $$
$$ \text{and} \qquad \Omega' = \left\{ x = (r,\theta) : R_1(\theta) < r < R_2(\theta), 0 < \theta < \frac{\pi}{2} \right\} $$
where $R_1$, $R_2$ are two Lipschitz functions $\left[0,\frac{\pi}{2}\right] \to (0,+\infty)$ with $\inf R_1 > 1$ and $\inf R_2 > \sup R_1$; we also suppose $R_1$ to be such that $R_1'$ is a function of bounded variation. Notice that we can choose $R_1$, $R_2$ such that the target domain $\Omega'$ is convex (for instance if the curves $r=R_1(\theta)$, $r=R_2(\theta)$ are actually two lines in the quarter plane). We assume that $f$ and $g$ are two Lipschitz densities on $\bar{\Omega}$, $\bar{\Omega'}$, bounded from above and below by positive constants, with the following hypothesis:
\begin{equation} \forall \theta \in \left(0,\frac{\pi}{2}\right), \, \int_0^1 f(r,\theta) \, r \dx r = \int_{R_1(\theta)}^{R_2(\theta)} g(r,\theta) \, r \dx r \label{egdensity} \end{equation}
which means that, for any $\theta$, the mass (with respect to $f$) of the segment joining the origin to the boundary of $\Omega$ and with angle $\theta$ is equal to the mass (with respect to $g$) of the segment with same angle joining the ``above'' and ``below'' boundaries of $\Omega'$ ({\it i.e.}~the curves $r=R_1(\theta)$ and $r=R_2(\theta)$) - see Figure~\ref{figomega} above.

Then the structure of the optimal maps for the Monge cost is given by the following:

\begin{prop} \label{prop1} Under these assumptions on $\Omega$, $\Omega'$, $f$, $g$, the Euclidean norm is a Kantorovich potential and the maximal transport rays are the segments joining $0$ to $(R_2(\theta),\theta)$. Consequently,
$$ T \in \O_1(\mu,\nu) \quad \iff \quad \left\{ \begin{array}{l} T_\# f = g \\[1mm] T(x) = \phi(x)\dfrac{x}{|x|} \; \text{for a.e.}~x \in \Omega  \end{array} \right. $$
for some function $\phi \in \Omega \to (\inf R_1,+\infty)$. \end{prop}

\begin{proof} For $\theta \in \left(0,\frac{\pi}{2}\right)$, we denote by $t(\cdot,\theta)$ a one-dimensional transport map from the measure $r \mapsto rf(r,\theta)$ to the measure $r \mapsto rg(r,\theta)$ (such a transport map exists since these two measures have same mass thanks to the equality \eqref{egdensity}).
It is then easy to check that the map
$$ T : x = (r,\theta) \in \Omega \to t(r,\theta) \frac{x}{|x|} $$
is a transport map from $f$ to $g$, and if we set $u=|\cdot|$, $u$ is of course $1$-Lipschitz and
$$ u(T(x))-u(x) = \left|t(r,\theta)\frac{x}{|x|}\right|-|x| = (|t(r,\theta)|-|x|) \frac{x}{|x|} =
\left|(t(r,\theta)-|x|)\frac{x}{|x|}\right|  = |T(x)-x| \qquad \text{for any } x\in \Omega.  $$
We deduce that that $u = |\cdot|$ is a Kantorovich potential. Consequently, a segment $[x,y]$ is a transport ray if and only if
$$ u(y)-u(x) = |y-x| \qquad {\it i.e.} \qquad |y-x|=|y|-|x|$$
Thus, we have $y = \lambda x$ for a positive $\lambda$; in such a case, $y$ and $x$ belong to the same line passing through the origin. In other words, the transport rays are included in the lines passing through the origin. Moreover, a transport map $T$ belongs to $\O_1(\mu,\nu)$ if and only if, for a.e.~$x~\in\Omega$, $ |T(x)-x| = |T(x)|-|x|$ which again means that $T(x) = \phi(x)\dfrac{x}{|x|}$ for some positive function $\phi$. \end{proof}

\noindent As we said in the introduction, this implies that any optimal map $T$ must present a singularity in $0$: formally, if we denote by $\phi(r,\theta)$ the norm of $T(x)$ for $x=(r,\theta) \in \Omega$, keeping in mind that $T(x)$ has same angle as $x$, we obtain that $0$ is sent to the whole curve $\theta\mapsto\phi(0,\theta)$ (we will see that this is actually true for the limit of a sequence of minimizers, see paragraph~\ref{parpresquegamma} below). In particular:

\begin{co} Under the above assumptions on $\Omega$, $\Omega'$, $\mu$, $\nu$, we have $ \O_1(\mu,\nu) \cap H^1(\Omega) = \emptyset $ \end{co}

\begin{proof} Let $T \in \O_1(\mu,\nu)$. By Prop.~\ref{prop1}, we have $T(x) = \phi(r,\theta)\dfrac{x}{|x|}$ for $x = (r,\theta)$, where $\phi$ is a real-valued function; the fact that $T$ sends $\mu$ onto $\nu$ implies $T(x) \in \Omega'$ for any $x$, thus $\phi$ is bounded from below on $\Omega$ by the lower bound of $R_1$. We now compute the Jacobian matrix of $T$. Denoting by $x^\perp$ the image of~$x$ by the rotation with angle $\pi/2$, we have in the basis $\left(\frac{x}{|x|},\frac{x^\perp}{|x|}\right)$:
$$ DT(x) = \frac{x}{|x|}\otimes\nabla \phi(x)-\frac{\phi(x)}{|x|}I_d  = \left( \begin{array}{cc} \partial_r\phi & 0 \\[2mm] \dfrac{\partial_\theta\phi}{r} & \dfrac{\phi}{r} \end{array} \right) $$
$$\text{thus} \qquad \int_\Omega |DT(x)|^2 \dx x \geq \int_0^1\int_0^{\frac{\pi}{2}} \frac{\phi(r,\theta)^2}{r^2} \, r \dx\theta\text{d}r \geq \int_0^{1} \frac{\pi}{2}\left( \inf R_1\right)^2 \frac{\text{d}r}{r} = +\infty \qedhere $$\end{proof}

\subsection{Heuristics} \label{sectheu}

In this paragraph, we give a preliminary example of analysis of the behavior of $J_\eps(T_\eps)$ when $\eps \to 0$ and $T_\eps$~approaches an optimal map $T$ for the Monge problem; this will not lead directly to a rigorous proof of the general result, but gives an idea of which quantities will appear.

Assume that $T \in \O_1(\mu,\nu)$ with $T(x) = \phi(r,\theta) \dfrac{x}{|x|}$, and let us build an approximation $(T_\eps)_\eps$ defined~by
$$ T_\eps(x) = \left\{ \begin{array}{ll} S(x) & \text{if } x \in \Omega_\delta \\ T(x) & \text{otherwise} \end{array}\right. $$
where $\delta$ will be fixed depending on $\eps$, $\Omega_\delta = B(0,\delta) \cap \Omega$ and $S$ will be build to send $(f\cdot\leb^d)|_{\Omega_\delta}$ onto the same image measure that the original $T$ has on $\Omega_\delta$ ($S$ will actually depend both on $\delta$ and on $\eps$, but we omit this dependence for the sake of simplicity of notations). In this case, we have
$$ J_\eps(T_\eps) - W_1 = \int_\Omega |T_\eps(x)-x| f(x) \dx x - W_1 + \eps \, \int_{\Omega_\delta} |DS|^2 + \eps\,\int_{\Omega\setminus \Omega_\delta} |DT|^2  $$
Since $T$ is optimal for the Monge problem and coincides with $T_\eps$ outside of $\Omega_\delta$, we have
$$ \int_{\Omega} |T_\eps(x)-x|f(x)\dx x - W_1 = \int_\Omega (|T_\eps(x)-x|-|T(x)-x|) f(x) \dx x = \int_{\Omega_\delta}(|S(x)-x|-|T(x)-x|) f(x) \dx x $$
We now claim that
$$ \int_{\Omega_\delta} |T(x)-x| f(x) \dx x = \int_{\Omega_\delta} (|S(x)|-|x|) f(x) \dx x $$
Indeed, we still have the equality $|T(x)-x| = |T(x)|-|x|$, and the image measures of $(f\cdot\leb^d)|_{\Omega_\delta}$ by $T$ and $S$ are the same. As a consequence,
$$ \int_{\Omega} |T_\eps(x)-x|f(x)\dx x - W_1  = \int_{\Omega_\delta} (|S(x)-x|-|S(x)|+|x|) f(x) \dx x $$
and, by the triangle inequality, $|S(x)-x|-|S(x)|+|x| \leq |(S(x)-x)-S(x)|+|x| \leq 2|x|$ so that
$$ \int_{\Omega} |T_\eps(x)-x|f(x)\dx x - W_1  \leq \int_{\Omega_\delta} 2|x| \dx x \leq 2\int_0^{\frac{\pi}{2}} \int_0^\delta r^2 \dx r \dx \theta \leq \frac{\pi\delta^3}{3}  $$
In order to estimate the norm of the Jacobian matrix $DT$ outside of $\Omega_\delta$, we recall that, in the basis~$\left(\dfrac{x}{|x|},\dfrac{x}{|x|}^{\!\perp}\right)$,
$$ DT(x) =  \left( \begin{array}{cc} \partial_r\phi & 0 \\[2mm] \dfrac{\partial_\theta\phi}{r} & \dfrac{\phi}{r} \end{array} \right) $$
$$ \text{so that} \qquad \int_{\Omega\setminus \Omega_\delta} |DT(x)|^2 \dx x = \int_\delta^1 \int_0^{\frac{\pi}{2}} \left( \frac{\phi(r,\theta)^2+\partial_\theta\phi(r,\theta)^2}{r} + r \partial_r\phi(r,\theta)^2 \right) \dx \theta \dx r $$ 
We now focus on the asymptotics of these integrals, which involve the partial derivatives of~$\phi$. First of all, we notice that, for each $\theta\in(0,\pi/2)$, the one-dimensional map $\phi(\cdot,\theta)$ sends the density $rf(\cdot,\theta)$ to the density $rg(r,\theta)$. Since the first of these densities is bounded from above and the second one from below, one can expect that the partial derivative~$\partial_r\phi(\cdot,\theta)$ is bounded (using the fact that the first density actually vanishes around~$0$, we even could prove that $\partial_r\phi(r,\theta) \to 0$ as~$r\to~0$). Then,
\begin{equation} \int_0^{\frac{\pi}{2}}\int_\delta^1 \partial_r\phi(r,\theta)^2 r \dx r \dx \theta \leq C \label{intdrphi} \end{equation}
where $C$ is a constant independent of $\delta$. On the other hand,
$$ \int_\delta^1 \int_0^{\frac{\pi}{2}}  \frac{\phi(r,\theta)^2+\partial_\theta\phi(r,\theta)^2}{r} \dx \theta \dx r = \int_\delta^1 ||\phi(r,\cdot)||_{H^1(0,\pi/2)}^2 \,\frac{\text{d}r}{r} $$
In this last integral, we make the change of variable $r = \delta^t$, which gives
$$ \int_\delta^1 \int_0^{\frac{\pi}{2}}  \frac{\phi(r,\theta)^2+\partial_\theta\phi(r,\theta)^2}{r} \dx \theta \dx r = |\log\delta| \int_0^1 ||\phi(\delta^t,\cdot)||_{H^1(0,\pi/2)}^2 \dx t $$
Assuming moreover that $||\phi(\delta^t,\cdot)||_{H^1(0,\pi/2)}^2 \to ||\phi(0,\cdot)||_{H^1(0,\pi/2)}^2$ as $\delta \to 0$, we formally get
\begin{equation} \eps\, \int_{\Omega\setminus \Omega_\delta} |DT|^2 = \eps|\log\delta| ||\phi(0,\cdot)||_{H^1(0,\pi/2)}^2 + \grando\limits_{\eps,\delta \to 0}(\delta^3+\eps) \label{intoutomegadelta} \end{equation}
It then remains to estimate the $L^2$-norm of the Jacobian matrix of $S$ on $\Omega_\delta$. Let us recall that $S$ has to be built so that $T_\eps$, defined on the whole $\Omega$, is still a transport map from $\mu$ to $\nu$ with finite Sobolev norm; thus, the map $S$, defined on $\Omega_\delta$, must send $\Omega_\delta$ onto its original image $S(\Omega_\delta)$ in a regular way while keeping the constraint on the image measures:
$$ S_\#(\mu|_{\Omega_\delta}) = T_\#(\mu|_{\Omega_\delta}) $$
Moreover, the regularity of the global map $T_\eps$ implies a compatibility condition at the boundary:
$$ S(x) = T(x) \qquad \text{for } |x|=\delta $$
Thanks to the Dacorogna-Moser's result, we are indeed able to build such a map $S$. Moreover, the diameter of $\Omega_\delta$ is $\sqrt{2}\delta$; on the other hand, $T(\Omega_\delta)$ contains the whole curve $\theta \mapsto \phi(0,\theta)$, so that its diameter is bounded from below by a positive constant independent of $\delta$. Thus, the best estimate that one can expect~is
$$ \Lip S \leq \frac{C}{\delta} $$
For a reasonable transport map $T$, one can show that such a map $S$ can be found with moreover $S(x)=T(x)$ for $|x|=\delta$ (see the paragraph \ref{construction} below). In this case, the global map $T_\eps$ still sends $\mu$ to $\nu$ and we~have
\begin{equation} \int_{\Omega_\delta} |DT|^2 \leq \int_{\Omega_\delta} \left(\frac{C}{\delta}\right)^2 \leq C^2 \label{intsomegadelta} \end{equation}
Combining \eqref{intdrphi}, \eqref{intoutomegadelta}, \eqref{intsomegadelta} leads then to
$$ J_\eps(T_\eps)-W_1 = \eps|\log\delta| ||\phi(0,\cdot)||_{H^1(0,\pi/2)} + \grando\limits_{\eps,\delta\to 0}(\delta^3+\eps)  $$
If we choose $\delta = \eps^{1/3}$, we obtain
$$ J_\eps(T_\eps) = W_1 + \eps|\log\eps| \frac{1}{3}||\phi(0,\cdot)||^2_{H^1(0,\pi/2)} + \grando\limits_{\eps\to 0}(\eps) $$
In particular:
\begin{itemize}
\item the first order of convergence of $J_\eps(T_\eps)$ to $W_1$ is not anymore $\eps$, but $\eps|\log\eps|$;
\item the first significant term only involves the behavior of $\phi$ around $0$, which is the common singularity of all the optimal transport maps (and the only crossing point of all the transport rays). Precisely, the asymptotics suggests that $T_\eps \to T$, where $T$ at $r=0$ minimizes $||\phi(r,\cdot)||_{H^1}$.
\end{itemize}

\subsection{Main result and consequences \label{parpresquegamma}}

The analysis in the above paragraph suggests to introduce the minimal value of $||\phi(0,\cdot)||_{H^1(0,\pi/2)}$ among the functions $\phi$ such that
$$ x \mapsto \phi(r,\theta)\frac{x}{|x|} $$
is a transport map from $\mu$ to $\nu$. In particular, for such a $\phi$ and for any $x \in \Omega$, the point $\phi(r,\theta) \dfrac{x}{|x|}$ still belongs to the target domain $\Omega'$; thus, its value at $r=0$ verifies
$$ \text{for a.e.}~\theta \in (0,\pi/2), \quad R_1(\theta) \leq \phi(0,\theta) \leq R_2(\theta) $$
We will thus set
$$ K = \min\left\{ \int_0^{\frac{\pi}{2}} (\phi(\theta)^2+\phi'(\theta)^2) \dx\theta \,:\, \phi \in H^1\left(0,\frac{\pi}{2}\right), \, R_1(\theta) \leq \phi(\theta) \leq R_2(\theta) \right\} $$
and call $\Phi$ the function which realizes the minimum (it is unique since the Sobolev norm is strictly convex and since the constraints $R_1\leq \phi\leq R_2$ define a convex set).

We notice also that, since, for any real-valued function $\phi$ with $R_1 \leq \phi$, we have
$$ ||\phi\wedge (\sup R_1)||^2_{H^1} \leq ||\phi||_{H^1}^2  $$
and since $\sup R_1$ is smaller than $\inf R_2$, we may remove, from the problem defining $K$, the constraint $\phi\leq R_2$; in other words, the following equality holds (and be used in the sequel):
$$ K = \min\left\{ \int_0^{\frac{\pi}{2}} (\phi(\theta)^2+\phi'(\theta)^2) \dx\theta \,:\, \phi \in H^1\left(0,\frac{\pi}{2}\right), \, R_1(\theta) \leq \phi(\theta) \right\} $$

Following the expansion we found in Paragraph~\ref{sectheu}, we are interested in the behavior, as $\eps\to 0$, of the functional
$$F_\eps:T\mapsto \frac{1}{\eps}\left( J_\eps(T_\eps)-W_1-\frac{K}{3}\eps|\log\eps| \right)$$

Let us recall quickly the statement of Theorem~\ref{maintheo}. We introduce the following notations:
$$ G:\phi\in H^1(0,\pi/2)\mapsto ||\phi||_{H^1(0,\pi/2)}^2-K $$
$$ \text{and} \qquad F(T) = \left\{ \begin{array}{l} +\infty \qquad \text{if } T \notin \O_1(\mu,\nu) \\[2mm] \displaystyle\int_0^1 G(\phi(r,\cdot)) \dfrac{\text{\textnormal{d}}r}{r} + \displaystyle\int_0^1 ||\partial_r\phi(r,\cdot)||_{L^2}^2 \, r \dx r \qquad \text{if } T \in \O_1(\mu,\nu), T(x) = \phi(r,\theta)\dfrac{x}{|x|} \end{array} \right. $$
The three statements we are interested in are the following:
\begin{enumerate}
\item For any family of maps $(T_\eps)_\eps$ such that $(F_\eps(T_\eps))_\eps$ is bounded, there exists a sequence $\eps_k \to 0$ and a map $T$ such that $T_{\eps_k} \to T$ in~$L^2(\Omega)$.
\item There exists a constant $C$, depending only on the domains $\Omega$, $\Omega'$ and of the measures $f$, $g$, so that, for any family of maps $(T_\eps)_{\eps>0}$ with $T_\eps \to T$ as $\eps \to 0$ in $L^2(\Omega)$, we~have
$$ \liminf\limits_{\eps \to 0} F_\eps(T_\eps) \geq F(T)-C $$
\item Moreover, there exists at least one family $(T_\eps)_{\eps > 0}$ such that $(F_\eps(T_\eps))_\eps$ is indeed bounded. \end{enumerate}
The proof of this result is given below in Section~\ref{mainproof}. We finish this section by some comments, and by the consequences on the behavior of $(J_\eps)_\eps$ and their minimizers.

\paragraph{A conjecture on the $\Gamma$-limit.} Notice that we have not stated here a complete $\Gamma$-convergence result, but we only provide an estimate on the $\Gamma$-liminf, and the existence of a sequence with equibounded energy. Actually, we conjecture that the $\Gamma$-limit of the sequence $F_\eps$ is exactly of the form $F-C$, for a suitable constant $C$ depending on the shape of $\Omega'$, and on $f(0)$ (again, the main important region is that around $x=0$ in $\Omega$, which must be sent on the curve~$\Phi$). We will give more details on this conjecture and on the value of the constant $C$ at the end of Paragraph \ref{gammaliminfest}).

However, we do not prove this result here; the estimate that we are really able to prove is enough to get interesting consequences on the minima and the minimizers of $F_\eps$.

\paragraph{Consequences on the minimal value of $J_\eps$.} If we apply the Theorem~\ref{maintheo} to a sequence $(T_\eps)_\eps$ where each $T_\eps$ minimizes $J_\eps$ (which is equivalent with minimizing $F_\eps$), we obtain that the sequence $(F_\eps(T_\eps))_\eps$ is bounded and
$$ \inf J_\eps = F_\eps(T_\eps) = W_1(\mu,\nu)+\frac{K}{3}\eps|\log\eps|+O(\eps) $$
We recover both the order $\eps|\log\eps|$ and the constant $K$ which appeared at the end of the above paragraph. Notice that a full knowledge of the $\Gamma$-limit would allow to compute the constant in the term $O(\eps)$.

\paragraph{Consequences on the behavior of $(T_\eps)_\eps$.} The qualitative consequences of Theorem~\ref{maintheo} come essentially from the following property of the functions $F$:

\begin{prop} \label{PhiL2fort} Let $T \in \O_1(\mu,\nu)$, $T(x) = \phi(r,\theta)\dfrac{x}{|x|}$, such that $F(T) < +\infty$. Then $r \mapsto \phi(r,\cdot)$ is continuous from $[0,1]$ to $L^2(0,\pi/2)$, and we have $\phi(0,\cdot)=\Phi$. \end{prop}

\noindent Combined with Theorem~\ref{maintheo}, Proposition \ref{PhiL2fort} implies that that if $(T_\eps)_\eps$ is a family of maps with $(F_\eps(T_\eps))_\eps$ bounded, then it has, up to a subsequence, a limit $T = \phi(r,\theta)\frac{x}{|x|}$, where $\phi(r,\theta)$ is continuous with respect to $r$ and has $\Phi(\theta)$ as limit as $r\to 0$. In other words, $T$ sends $0$ onto the curve $r = \Phi(\theta)$ which has the best $H^1$-norm among the curves with values in the target domain $\Omega'$. This is in particular true if each $T_\eps$ minimizes $J_\eps$ (thus $F_\eps$). \smallskip

In order to prove Proposition~\ref{PhiL2fort}, the following lemma will be needed:

\begin{lem} \label{phimoinsphi} The function $\Phi$ is Lipschitz on $(0,\pi/2)$ and, for any $\phi \in H^1(0,\pi/2)$ verifying $R_1 \leq \phi$, we have
$$G(\phi) = ||\phi||^2_{H_1}-||\Phi||^2_{H_1} \geq ||\phi-\Phi||_{H^1}^2 $$ \end{lem}


\begin{proof} From the Euler-Lagrange equation associated to the problem which defines $\Phi$, we infer that $\Phi'' = \Phi$ on the set of points where $R_1 < \Phi < R_2$; this, together with the information that $R_1$ and $R_2$ are both Lipschitz, implies that $\Phi$ is Lipschitz as well. On the other hand, denoting by $\mathcal{C}$ the (convex) set of Sobolev functions which are between $R_1$ and $R_2$, we notice that $\Phi$ is actually the orthogonal projection of $0$ onto $\mathcal{C}$ in the Hilbert space $H^1(0,\pi/2)$; in particular,
\begin{equation} \forall \phi \in \mathcal{C}, \, \langle \Phi,\phi-\Phi \rangle \geq 0 \label{projh1} \end{equation}
If now $\phi \in \mathcal{C}$, then
$$ G(\phi) - ||\phi-\Phi||_{H^1}^2 = ||\phi||_{H^1}^2-||\Phi||_{H^1}^2- ||\phi-\Phi||_{H^1}^2 = 2\, \langle \Phi,\phi-\Phi \rangle $$
which is non-negative thanks to the inequality \eqref{projh1}. \end{proof}

\begin{proof}[Proof of Prop.~\ref{PhiL2fort}] The assumption on $T$ implies that the integrals
$$\displaystyle\int_0^1 G(\phi(r,\cdot)) \,\frac{\text{d}r}{r} \qquad \text{and} \qquad \displaystyle\int_0^1 ||\partial_r\phi(r,\cdot)||_{L^2(0,\pi/2)}^2 \, r \dx r$$
are both controlled by some finite constant $A$. Now we have for $\theta \in (0,\pi/2)$:
$$ |\phi(r_1,\theta)-\phi(r_2,\theta)| = \int_{r_1}^{r_2} \partial_r\phi(r,\theta) \dx r \leq \left(\int_{r_1}^{r_2} \partial_r\phi(r,\theta)^2 \, r \dx r \right)^{1/2} \, \left(\int_{r_1}^{r_2} \frac{\text{d}r}{r} \right)^{1/2} $$
$$\text{thus} \qquad \int_0^{\pi/2} |\phi(r_1,\theta)-\phi(r_2,\theta)|^2 \dx\theta \leq \left(\int_{r_1}^{r_2} ||\partial_r\phi(r,\cdot)||_{L^2(0,\pi/2)}^2 \, r \dx r\right) \left( \int_{r_1}^{r_2} \frac{\text{d}r}{r} \right) $$
\begin{equation} \text{so that} \qquad ||\phi(r_1,\cdot)-\phi(r_2,\cdot)||_{L^2(0,\pi/2)}^2 \dx\theta \leq A\log\frac{r_2}{r_1} \label{loglip} \end{equation}
This proves the continuity of $r \mapsto \phi(r,\cdot)$.

On the other hand, thanks to the Lemma~\ref{phimoinsphi}, we have
$$ A \geq \int_0^1 G(\phi(r,\cdot)) \frac{\text{d}r}{r} \geq \int_0^1 ||\phi(r,\cdot)-\Phi||_{L^2(0,\pi/2)}^2 \frac{\text{d}r}{r} $$
By setting $r = e^{-t}$, we obtain
$$ A \geq \int_0^{+\infty} ||\phi(e^{-t},\cdot)-\Phi||_{L^2}^2 \dx t $$
But, for $t_1 < t_2 \in (0,+\infty)$, we have
$$ \begin{array}{ll} \Big| ||\phi(e^{-t_1},\cdot)-\Phi||_{L^2}^2-||\phi(e^{-t_2},\cdot)-\Phi||_{L^2}^2 \Big| & = \Big|\langle \phi(e^{-t_1},\cdot)-\phi(e^{-t^2},\cdot), \Phi \rangle_{L^2}\Big| \\[3mm]
& \leq ||\phi(e^{-t_1},\cdot)-\phi(e^{-t_2},\cdot)||\,||\Phi|| \\[1mm]
& \leq A\log\dfrac{e^{-t_2}}{e^{-t_1}} = A(t_2-t_1) \end{array} $$
where the last inequality comes from \eqref{loglip}. Thus, the function $t \mapsto ||\phi(e^{-t},\cdot)-\Phi||_{L^2}^2$ is Lipschitz and belongs to $L^1(0,+\infty)$. This implies that it vanishes at $+\infty$, so that $\phi(r,\cdot) \to \Phi$ in $L^2(0,\pi/2)$ as $r \to 0$. \end{proof}

\section{Proof of Theorem \ref{maintheo} \label{mainproof}}

\subsection{$\Gamma$-liminf estimate \label{gammaliminfest}}

First of all, given a family $(T_\eps)_\eps$ of transport maps, let us write precisely the expression of $F_\eps(T_\eps)$. We~have
$$F_\eps(T_\eps) = \frac{1}{\eps} \left(\int_\Omega|T_\eps(x)-x|f(x)\dx x - W_1\right) + \int_\Omega |DT_\eps(x)|^2\dx x - \frac{K}{3}|\log\eps|  $$
We decompose $T_\eps$ into radial and tangential components
$$ T_\eps(x) = \phi_\eps(r,\theta)\frac{x}{|x|}+\psi_\eps(r,\theta)\frac{x}{|x|}^{\!\perp} $$
and compute
$$ DT_\eps = \frac{x}{|x|}\otimes\nabla\phi_\eps(x)+\frac{\phi_\eps(x)}{|x|}\left(I_d-\frac{x}{|x|}\otimes\frac{x}{|x|}\right)+\frac{x}{|x|}^{\!\perp}\otimes\nabla\psi_\eps(x)+\frac{\psi_\eps(x)}{|x|}\left(R-\frac{x}{|x|}^{\!\perp}\otimes\frac{x}{|x|}\right) $$
where $R$ denotes the rotation with angle $\pi/2$ and we still set $x^\perp = Rx$. Thus, the matrix of $DT_\eps$ in the basis~$\left(\dfrac{x}{|x|},\dfrac{x}{|x|}^{\!\perp}\right)$~is
$$ DT_\eps(x) = \left(\begin{array}{cc} \partial_r\phi_\eps & \partial_r\psi_\eps \\[2mm] \dfrac{\partial_\theta\phi_\eps-\psi_\eps}{r} & \dfrac{\phi_\eps+\partial_\theta\psi_\eps}{r} \end{array} \right) $$
so that
$$ |DT_\eps|^2 = \partial_r\phi_\eps^2+\partial_r\psi_\eps^2+\frac{(\partial_\theta\phi_\eps-\psi_\eps)^2}{r^2}+\frac{(\phi_\eps+\partial_\theta\psi_\eps)^2}{r^2} $$
Setting $\delta =\eps^{1/3}$, we get
$$ \int_\Omega |DT_\eps|^2 = \int_{\Omega_\delta} |DT_\eps|^2 + \int_\delta^1 (||\partial_r\phi_\eps(r,\cdot)||_{L^2}^2+||\partial_r\psi_\eps||_{L^2}^2) \, r \dx r + \int_\delta^1 (||\partial_\theta\phi_\eps-\psi_\eps||_{L^2}^2+||\phi_\eps+\partial_\theta\psi_\eps||_{L^2}^2) \,\frac{\text{d}r}{r} $$
On the other hand, we already know that
$$ \int_{\Omega} |T_\eps(x)-x|f(x)\dx x - W_1 = \int_{\Omega} (|T_\eps(x)-x|-|T_\eps(x)|+|x|) f(x) \dx x $$
and notice that, by definition of $\delta$,
$$ \frac{K}{3}|\log\eps| = K|\log\delta| = K\int_\delta^1 \frac{\text{d}r}{r} $$
Finally, the complete expression of $F_\eps$ is the following:
\begin{multline*} F_\eps(T_\eps) = \frac{1}{\eps} \int_\Omega (|T_\eps(x)-x|-|T_\eps(x)|+|x|) f(x) \dx x + \int_{\Omega_\delta} |DT_\eps|^2 \\ + \int_\delta^1 (||(\partial_\theta\phi_\eps-\psi_\eps)(r,\cdot)||_{L^2}^2+||(\phi_\eps+\partial_\theta\psi_\eps)(r,\cdot)||_{L^2}^2 - K) \,\frac{\text{d}r}{r} + \int_\delta^1 (||\partial_r\phi_\eps(r,\cdot)||_{L^2}^2+||\partial_r\psi_\eps||_{L^2}^2) \, r \dx r  \end{multline*}
thus, if we denote by $H(\phi,\psi) = \displaystyle\int_0^{\pi/2} (\phi'(\theta)-\psi(\theta))^2+(\phi(\theta)+\psi'(\theta))^2 \dx\theta$ for $\phi,\psi \in H^1(0,\pi/2)$, we have
\begin{multline} F_\eps(T_\eps) = \frac{1}{\eps} \int_\Omega (|T_\eps(x)-x|-|T_\eps(x)|+|x|) f(x) \dx x + \int_{\Omega_\delta} |DT_\eps|^2 \\ + \int_\delta^1 (H(\phi_\eps(r,\cdot),\psi_\eps(r,\cdot)) - K) \,\frac{\text{d}r}{r} + \int_\delta^1 (||\partial_r\phi_\eps(r,\cdot)||_{L^2}^2+||\partial_r\psi_\eps||_{L^2}^2) \, r \dx r \label{fepsteps} \end{multline}
The following lemma collects some properties of the function $H$.

\begin{lem} The function $H$, defined on $H^1(0,\pi/2)^2$, satisfies the following properties:
\begin{itemize}
\item $H$ is lower semi-continuous with respect to the strong $L^2$-convergence;
\item Assume that $(\phi,\psi)$ satisfies, for any $\theta$,
$$ \phi(\theta)\hat{x}(\theta)+\psi(\theta)\hat{x}^\perp(\theta) \in \Omega' $$
where $\hat{x}(\theta) = (\cos\theta,\sin\theta)$. We denote by $\tilde{\phi}(\theta) = \max(\phi(\theta),R_1(\theta))$. Then we have the inequality
\begin{equation} H(\phi,\psi) \geq K+\frac{1}{2}||\tilde{\phi}-\Phi||_{L^2}^2-B||\psi||_{L^2(0,\pi/2)}^{2/3} \label{deuxtiers} \end{equation}
for some positive constant $B$ which only depends on $\Omega'$.
\end{itemize} \end{lem}

\begin{proof}{ }

{\it \underline{Step 1:} the semi-continuity of $H$.} We take a sequence $(\phi_n,\psi_n)_n$ converging to some $(\phi,\psi)$ for the $L^2$-norm. Up to subsequences, we can assume that
$$ \liminf\limits_{n\to+\infty} H(\phi_n,\psi_n) = \lim\limits_{n\to+\infty} H(\phi_n,\psi_n) $$
and we also assume that $(H(\phi_n,\psi_n))_n$ is bounded. Now we remark that
$$ H(\phi_n,\psi_n) = ||\phi_n'-\psi_n||_{L^2}^2+||\phi_n+\psi_n'||_{L^2}^2 \geq (||\phi_n'||_{L^2}-||\psi_n||_{L^2})^2+(||\phi_n'||_{L^2}-||\psi_n||_{L^2}^2)^2 $$
$$ \text{thus}\qquad ||\phi'_n||_{L^2} \leq \sqrt{H(\phi_n,\psi_n)}+||\psi_n||_{L^2} \qquad \text{and} \qquad ||\psi'_n||_{L^2} \leq \sqrt{H(\phi_n,\psi_n)}+||\phi_n||_{L^2}$$
We deduce that $(\phi_n)_n$, $(\psi_n)_n$ are bounded in $H^1(0,\pi/2)$ so that the convergence $(\phi_n,\psi_n)\to(\phi,\psi)$ actually holds, up to a subsequence, weakly in $H^1(0,\pi/2)$. Now the convexity of $(\phi,\phi',\psi,\psi')\mapsto(\phi'-\psi)^2+(\phi+\psi')^2$ implies that $H$ is lower semi-continuous with respect to the weak convergence in $H^1(0,\pi/2)$, which allows to conclude.

Now we pass to the proof of the inequality \eqref{deuxtiers}. We begin by a kind of ``sub-lemma'' which will be useful several times in the proof:

{\it \underline{Step 2:} preliminary estimates.} We claim that:
\begin{itemize}
\item for any $t\in(0,\pi/2)$, we have the inequality
\begin{equation} 0\leq h(t)\leq B_1|\psi(t)| \label{hpsi} \end{equation}
for some constant $B_1$ depending only on $\Omega'$;
\item we have the inequality
\begin{equation} \left|\langle \tilde{\phi},h \rangle_{H^1}\right| \leq B_2||h||_\infty \label{estimscal} \end{equation}
for some constant $B_2$ depending only on $\Omega'$ (recall here that $\tilde\phi = \max(R_1,\phi)$;
\item these two inequalities lead to the estimate
\begin{equation} ||\phi||_{H^1}^2 \geq K+||\tilde{\phi}-\Phi||_{H^1}^2 - B_3||\psi||_\infty \label{phih1} \end{equation}
for some constant $B_3$ depending only on $\Omega'$.
\end{itemize}

\noindent First, we remark that the constraint $\phi(\theta)\hat{x}(\theta)+\psi(\theta)\hat{x}^\perp(\theta) \in \Omega'$ 
implies
$$ R_1(\theta')^2 < \phi(\theta)^2+\psi(\theta)^2 < R_2(\theta')^2 \qquad \text{where} \qquad \theta'=\theta+\arcsin\dfrac{\psi(\theta)}{\sqrt{\phi(\theta)^2+\psi(\theta)^2}} $$
Thus, we have
$$ \begin{array}{ll} h(\theta) & = R_1(\theta)-\phi(\theta) \\[1mm]
 & = R_1(\theta)-R_1(\theta')+R_1(\theta')-\phi(\theta) \\[1mm]
 & \leq (\Lip R_1) |\theta-\theta'| + \sqrt{\phi^2(\theta)+\psi^2(\theta)}-\phi(\theta) \\[1mm]
 & \leq (\Lip R_1) \arcsin\dfrac{|\psi(\theta)|}{R_1(\theta')} + |\psi(\theta)| \\[1mm]
 & \leq \left(\dfrac{\pi}{2}\dfrac{\Lip R_1}{\inf R_1} +1 \right) \, |\psi(\theta)| \end{array}$$
which is \eqref{hpsi} with $B_1 = \dfrac{\pi}{2}\dfrac{\Lip R_1}{\inf R_1} +1$.

Second, we recall that $h=(R_1-\phi)^+$, thus $\phi+h=R_1$ on any point where $h \neq 0$. This leads to
$$ \left|\int_0^{\pi/2} (\phi+h)h\right| = \left| \int_0^{\pi/2} R_1 h \right| \leq \frac{\pi}{2} (\sup R_1) ||h||_\infty   $$
$$ \text{and} \qquad \left|\int_0^{\pi/2} (\phi+h)'h'\right| = \left|\int_0^{\pi/2} R_1'h'\right| = \left| [R_1'h]_0^{\pi/2}-\int_0^{\pi/2} R_1''h\right| \leq (2\sup R_1'+||R_1''||_1) ||h||_\infty $$
We get \eqref{estimscal} with $B_2=\left(\dfrac{\pi}{2}\sup R_1 + 2\Lip R_1 + ||R_1''||_{L^1}\right)$.

Third, we write
\begin{equation} ||\phi||_{H^1}^2 = ||\tilde{\phi}||_{H^1}^2 + ||h||_{H^1}^2+2\langle\tilde{\phi},h\rangle \label{phiphitildeh} \end{equation}
Since $\tilde{\phi} \geq R_1$ on $(0,\pi/2)$ and thanks to the Lemma \ref{phimoinsphi}, we have $||\tilde{\phi}||_{H^1}^2 \geq ||\tilde{\phi}-\Phi||_{H^1}^2+K$. On the other hand, by using \eqref{hpsi} and \eqref{estimscal}, we have
$$ \langle\tilde{\phi},h\rangle \geq -B_2||h||_\infty \geq -B_1B_2||\psi||_\infty $$
We insert into \eqref{phiphitildeh} and skip $||h||_{H^1}^2$ since it is non-negative to get
$$ ||\phi||_{H^1}^2 \geq K+||\tilde{\phi}-\Phi||_{H^1}^2 - 2B_1B_2||\psi||_\infty  $$
thus \eqref{phih1} holds with $B_3=2B_1B_2$.

{\it \underline{Step 3:} the inequality \eqref{deuxtiers} holds if $||\phi'||_{L^2}$ is large enough.} We start from
$$ H(\phi,\psi) = ||\phi||_{H^1}^2+||\psi||_{H^1}^2 + 2 \int_0^{\pi/2} \phi\psi'-2\int_0^{\pi/2}\psi\phi' 
 = ||\phi||_{H^1}^2+||\psi||_{H^1}^2-4\int_0^{\pi/2} \psi\phi' - 2[\phi\psi]_{0}^{\pi/2} $$
First, the condition on $(\phi,\psi)$ implies that $||\phi||_\infty,||\psi||_\infty \leq \sup{R_2}$ so that
$$ |[\phi\psi]_{0}^{\pi/2}| \leq 2(\sup{R_2})^2 $$
On the other hand,
$$ \left|\int_0^{\pi/2} \psi\phi'\right| \leq ||\psi||_\infty \sqrt{\frac{\pi}{2}} ||\phi'||_{L^2} \leq \sup{R_2} \sqrt{\frac{\pi}{2}} ||\phi'||_{L^2} $$
This leads to
$$ H(\phi,\psi) \geq ||\phi||_{H^1}^2+||\psi||_{H^1}^2 - 4\sup{R_2} \sqrt{\frac{\pi}{2}} ||\phi'||_{L^2} - 4(\sup{R_2})^2 $$
$$ \geq \frac{1}{2}||\phi||_{H^1}^2+\left(\frac{1}{2}||\phi'||_{L^2}^2-4\sup{R_2}\sqrt{\frac{\pi}{2}} ||\phi'||_{L^2} - 4(\sup{R_2})^2\right) $$
By using \eqref{phih1}, we obtain
$$  H(\phi,\psi) \geq \frac{1}{2}(K+||\tilde{\phi}-\Phi||_{H^1}^2-B_3||\psi||_\infty)+\left(\frac{1}{2}||\phi'||_{L^2}^2-4\sup{R_2}\sqrt{\frac{\pi}{2}} ||\phi'||_{L^2} - 4(\sup{R_2})^2\right)$$
and, since $|\psi| \leq \sqrt{\phi^2+\psi^2} \leq R_1$, we have
$$ H(\phi,\psi) \geq \frac{1}{2}(K+||\tilde{\phi}-\Phi||_{H^1})+\left(\frac{1}{2}||\phi'||_{L^2}^2-4\sup{R_2}\sqrt{\frac{\pi}{2}} ||\phi'||_{L^2} - \left(4(\sup{R_2})^2+\frac{B_3}{2}\sup R_2\right)\right)  $$
The announced estimate \eqref{deuxtiers} holds as soon as the term in brackets is greater that $\dfrac{K}{2}$, which is true provided that $||\phi'||_{L^2} \geq B_4$ where $B_4$ is the largest root of the polynom
$$ \frac{1}{2}X^2-4\sup~R_2\sqrt{\frac{\pi}{2}}X-\left(4(\sup{R_2})^2+\frac{B_3}{2}\sup R_2+\frac{K}{2}\right) $$
and $B_4$ only depends of $\Omega'$.

{\it \underline{Step 4:} case $||\phi'||_{L^2} \leq B_4$.} In this case, we still have
$$ H(\phi,\psi) = ||\phi||_{H^1}^2+||\psi||_{H^1}^2 -4\int_0^{\pi/2} \psi\phi' - 2[\phi\psi]_{0}^{\pi/2} $$
$$ \text{with} \qquad \left|\int_0^{\pi/2} \psi\phi'\right| \leq ||\psi||_{\infty}\sqrt{\frac{\pi}{2}} ||\phi'||_{L^2} \leq \sqrt{\frac{\pi}{2}} B_4 ||\psi||_\infty $$
$$ \text{and} \qquad |[\phi\psi]_{0}^{\pi/2}| \leq 2||\phi||_\infty||\psi||_\infty \leq 2\sup{R_2} ||\psi||_\infty  $$
This leads to
$$ H(\phi,\psi) \geq ||\phi||_{H^1}^2+||\psi||_{H^1}^2 - \left(\sqrt{\frac{\pi}{2}} B_4+2\sup{R_2}\right)||\psi||_\infty $$
$$ \geq K+||\tilde{\phi}-\Phi||_{H^1}^2+||\psi||_{H^1}^2-B_5||\psi||_\infty $$
where we have again used \eqref{phih1} and set $B_5=\left(\sqrt{\frac{\pi}{2}} B_4+2\sup{R_2}\right)+B_3$, which only depends on $\Omega'$.

It now remains to estimate $||\psi||_{H^1}^2-B_5||\psi||_\infty$ from below with $-||\psi||_{L^2}^{2/3}$. The condition on $(\phi,\psi)$ implies that $\psi(0) \geq 0$ and $\psi(\pi/2) \leq 0$, so that there exists $t_0$ such that $\psi(t_0) = 0$. We then have
$$ \psi^2(t) = \int_{t_0}^t \ddt (\psi^2) = \int_{t_0}^t 2\psi\psi' \leq 2||\psi||_{L^2}||\psi'||_{L^2} \quad \text{thus} \quad ||\psi||_\infty \leq \sqrt{2} \sqrt{||\psi||_{L^2}||\psi'||_{L^2}} $$
We use the Young inequality
$$ab \leq \frac{(\alpha a)^p}{p}+\frac{(b/\alpha)^q}{q}  \qquad \text{for } \frac{1}{p}+\frac{1}{q} = 1 \text{ and } \alpha > 0$$
with $p=4$, $q=4/3$, $a = \sqrt{||\psi'||_{L^2}}$ and $b = \sqrt{||\psi||_{L^2}}$, to get
$$ ||\psi||_\infty \leq \frac{\sqrt{2}\alpha^4}{4} ||\psi'||_{L^2}^2 + \frac{3\sqrt{2}}{4\alpha^{4/3}} ||\psi||_{L^2}^{2/3}  $$
We deduce
$$ ||\psi'||_{H^1}^2-B_5||\psi||_\infty \geq \left(1-\frac{\sqrt{2}B_5\alpha^4}{4}\right)||\psi'||_{L^2}-\frac{3\sqrt{2}B_5}{4\alpha^{4/3}}||\psi||_{L^2}^{2/3}  $$
By choosing $\alpha$ so that $\dfrac{\sqrt{2}B_5\alpha^4}{4} = 1$, we obtain
$$||\psi||_{H^1}^2 -B_5||\psi||_\infty \geq -B||\psi||_{L^2}^{2/3} $$
where $B = \dfrac{3\sqrt{2}B_5}{4\alpha^{4/3}}$ depends only on $\Omega'$ and $K$. This achieves the proof.\end{proof}

\noindent We will also need the following estimate on the first term of the expression \eqref{fepsteps}.

\begin{lem} Let $T$ be a transport map from $\mu$ to $\nu$. We write
$$ T(x) = \phi(x) \frac{x}{|x|} + \psi(x) \frac{x}{|x|}^{\!\perp} $$
Then, for a.e.~$x$,
\begin{equation} |T(x)-x|-|T(x)|+|x| \geq A |x| \psi^2(x) \label{xpsi2} \end{equation}
for some constant $A$ which only depends on $\Omega'$.\end{lem}

\begin{proof} We compute:
$$ |T(x)-x|-|T(x)|+|x| = \frac{|T(x)-x|^2-|T(x)|^2}{|T(x)-x|+|T(x)|} + |x| $$
We have $|T(x)-x|^2 = (\phi(x)-|x|)^2+\psi(x)^2$ and $|T|^2=\phi^2+\psi^2$, so that
$$ |T(x)-x|-|T(x)|+|x| = \frac{|x|^2-2|x|\phi(x)}{|T(x)-x|+|T(x)|}+|x| = |x|\,\frac{|x|+|T(x)-x|+|T(x)|-2\phi(x)}{|T(x)-x|+|T(x)|}$$
We remark that
$$ |T(x)-x|-\phi(x)+|x| = \sqrt{(\phi(x)-|x|)^2+\psi(x)^2}-(\phi(x)-|x|) \geq 0 $$
$$ \text{thus} \quad |T(x)-x|-|T(x)|+|x| \geq |x| \frac{|T(x)|-\phi(x)}{|T(x)-x|+|T(x)|} = |x| \frac{|T(x)|^2-\phi(x)^2}{(|T(x)-x|+|T(x)|)(|T(x)|+\phi(x))}  $$
Since $x \in \Omega$ and $T(x) \in \Omega'$, we have
$$ |T(x)-x|+|T(x)| \leq 2|T(x)|+|x| \leq 2\sup{R_2} + 1 $$
$$ \text{and} \qquad |T(x)|+\phi(x) \leq 2|T(x)| \leq 2\sup{R_2} $$
On the other hand, $|T(x)|^2-\phi(x)^2 = \psi(x)^2$. This leads to the result with $ A = \dfrac{1}{(2\sup{R_2}+1)(2\sup{R_2})}  $ \end{proof}

\noindent The estimate \eqref{xpsi2} leads to
$$ \int_{\Omega}(|T_\eps(x)-x|-|T_\eps(x)|+|x|) f(x) \dx x \geq A\inf{f}\, \int_0^1 ||\psi_\eps(r,\cdot)||_{L^2}^2 \, r \dx r $$
and the estimate \eqref{deuxtiers} to
$$ \int_\delta^1 (H(\phi_\eps(r,\cdot),\psi_\eps(r,\cdot))-K) \frac{\text{d}r}{r} \geq  \int_\delta^1 \left(-B||\psi_\eps(r,\cdot)||_{L^2}^{2/3}+\frac{1}{2}||\tilde{\phi}_\eps(r,\cdot)-\Phi||_{H^1}^2\right) \, \frac{\text{d}r}{r}  $$
where we again have set $\tilde{\phi}_\eps=\max(R_1,\phi_\eps)$. By inserting into \eqref{fepsteps}, we have
\begin{multline} F_\eps(T_\eps) \geq \frac{A\inf{f}}{\eps} \int_0^1 ||\psi_\eps(r,\cdot)||_{L^2}^2 \, r^2 \dx r - B\int_\delta^1||\psi_\eps(r,\cdot)||_{L^2}^{2/3} \frac{\text{d}r}{r} 
\\ + \frac{1}{2}\,\int_\delta^1 (H(\phi_\eps(r,\cdot),\psi_\eps(r,\cdot))-K+B||\psi_\eps(r,\cdot)||_{L^2}^{2/3}) \,\frac{\text{d}r}{r} + \int_\delta^1 ||\partial_r\phi_\eps(r,\cdot)||_{L^2}^2 \, r \dx r \label{fepstepsbis} \end{multline}
Let us denote by $X_\eps = \dfrac{1}{\eps} \displaystyle\int_0^1 ||\psi_\eps(r,\cdot)||_{L^2}^2 r^2 \dx r$. By the H\"older inequality applied with respect to the measure with density $1/r$ on $(\delta,1)$, we have
$$ \int_\delta^1||\psi_\eps(r,\cdot)||_{L^2}^{2/3} \,\frac{\text{d}r}{r} = \int_\delta^1(||\psi_\eps(r,\cdot)||_{L^2}^2 r^3)^{1/3} \, \frac{1}{r} \, \frac{\text{d}r}{r} \leq \left(\int_\delta^1 ||\psi_\eps||_{L^2}^2 r^3 \,\frac{\text{d}r}{r} \right)^{1/3} \, \left(\int_\delta^1 \frac{1}{r^{4/3}} \,\frac{\text{d}r}{r} \right)^{3/4}  $$
$$ \text{with} \qquad  \int_\delta^1 ||\psi_\eps||_{L^2}^2 r^3 \,\frac{\text{d}r}{r} \leq \eps X_\eps \quad \text{and} \quad \int_\delta^1 \frac{1}{r^{4/3}} \,\frac{\text{d}r}{r} = \frac{3}{4}\left(\frac{1}{\delta^{4/3}}-1\right) \leq \frac{3}{2\delta^{4/3}}  $$
which leads to
$$ \int_\delta^1||\psi_\eps(r,\cdot)||_{L^2}^{2/3} \,\frac{\text{d}r}{r} \leq (\eps X_\eps)^{1/3} \left(\frac{3}{2\delta^{4/3}}\right)^{3/4} =  \sqrt{\frac{3\sqrt{3}}{2\sqrt{2}}}\, X_\eps^{1/3}  $$
since $\delta = \eps^{1/3}$. We insert into \eqref{fepstepsbis} to obtain
\begin{multline} F_\eps(T_\eps) \geq (A\inf{f}) X_\eps -B' \,X_\eps^{1/3} \\
+ \int_\delta^1 (H(\phi_\eps(r,\cdot),\psi_\eps(r,\cdot))-K+B||\psi_\eps(r,\cdot)||_{L^2}^{2/3}) \,\frac{\text{d}r}{r} + \int_\delta^1 ||\partial_r\phi_\eps(r,\cdot)||_{L^2}^2 \, r \dx r \label{fepsteps2} \end{multline}
where $B' = \sqrt{\frac{3\sqrt{3}}{2\sqrt{2}}} \, B$.

Let us assume that $(F_\eps(T_\eps))_\eps$ is bounded by a positive constant $M$. This implies that $(X_\eps)_\eps$ is bounded by some constant $M'$ (otherwise the term $(A\inf{f}) X_\eps -B' \,X_\eps^{1/3}$ would be unbounded as $\eps\to 0$, and the other term is positive), thus
$$ \int_0^1 ||\psi_\eps(r,\cdot)||_{L^2}^2 r^2 \dx r \leq M'\eps $$
and $\psi_\eps \to 0$ a.e.~on $\Omega$. Since $(X_\eps)_\eps$ and $(F_\eps(T_\eps))_\eps$ are bounded, \eqref{fepsteps2} provides
$$ \int_\delta^1 (H(\phi_\eps(r,\cdot),\psi_\eps(r,\cdot)-K+B||\psi_\eps(r,\cdot)||_{L^2}^{2/3}) \,\frac{\text{d}r}{r} + \int_\delta^1 ||\partial_r\phi_\eps(r,\cdot)||_{L^2}^2 \, r \dx r \leq M'' $$
for some constant $M''$ which does not depend on $\eps$. We now use the estimate \eqref{deuxtiers} to get
$$ \frac{1}{2}\int_\delta^1 ||\tilde{\phi}_\eps-\Phi||_{H^1}^2\, \frac{\text{d}r}{r} + \int_\delta^1 ||\partial_r\phi_\eps(r,\cdot)||_{L^2}^2 \, r \dx r \leq M'' $$
We thus have a $L^2$-loc bound on $\partial_r\phi_\eps$ and on $\partial_\theta\tilde{\phi}_\eps$, but since $\tilde{\phi}(r,\theta)=\max(R_1(\theta),\phi(r,\theta))$, the bound on $\partial_r\phi_\eps$ implies a bound on $\partial_r\tilde{\phi}_\eps$. Therefore, the family $(\tilde{\phi}_\eps)_\eps$ is bounded in $H^1_{loc}(\Omega)$, and then there exists $\eps_k \to 0$ and $\tilde{\phi}$ such that $\tilde{\phi}_\eps \to \tilde{\phi}$ a.e.~on $\Omega$. But we recall that the estimation \eqref{hpsi} still holds and provides
$$ |\phi_{\eps_k}-\tilde{\phi}_{\eps_k}| \leq B_1|\psi_\eps| \to 0 $$
which leads to $\phi_{\eps_k} \to \tilde{\phi}$ a.e.~on $\Omega$. If we set now $T(x)=\tilde{\phi}\dfrac{x}{|x|}$, we have proven that $T_{\eps_k} \to T$ a.e.~on~$\Omega$; since $|T_\eps| \leq \sup R_2$ for any $\eps$, this convergence also holds in $L^2(\Omega)$. This proves the first statement of Theorem~\ref{maintheo}.

Assume now that $(T_\eps)_\eps$ is a family of transport maps converging to some $T$ for the $L^2$-norm on $\Omega$. We deduce from \eqref{fepsteps2} that, if we set $C=- \inf\limits_{X>0} (A\inf{f}X^3-B'X)$, which only depends on $\Omega'$, we have
\begin{equation} F_\eps(T_\eps) \geq -C+\int_\delta^1 (H(\phi_\eps(r,\cdot),\psi_\eps(r,\cdot))-K+B||\psi_\eps(r,\cdot)||_{L^2}^{2/3}) \,\frac{\text{d}r}{r} + \int_\delta^1 ||\partial_r\phi_\eps(r,\cdot)||_{L^2}^2 \, r \dx r \label{fepstepsC} \end{equation}
Assuming that $(F_\eps(T_\eps))_\eps$ is bounded, the above computations give a $H^1$-loc bound for $(\phi_\eps)_\eps$, thus
$$ \liminf\limits_{\eps \to 0} \int_\delta^1 ||\partial_r\phi_\eps(r,\cdot)||_{L^2}^2 \, r \dx r \geq \int_0^1 ||\partial_r\phi(r,\cdot)||_{L^2}^2 \, r \dx r $$
since this functional is lower semi-continuous for the weak convergence in $H^1(\Omega)$. On the other hand, the semi-continuity of $H$ provides
$$ \liminf\limits_{k \to +\infty} (H(\phi_\eps(r,\cdot),\psi_\eps(r,\cdot))-K+B||\psi_\eps(r,\cdot)||_{L^2}^{2/3}) \geq G(\phi(r,\cdot)) $$
for a.e.~$r\in (0,1)$, and the estimate \eqref{deuxtiers} shows also that
$$ H(\phi_\eps(r,\cdot),\psi_\eps(r,\cdot))-K+B||\psi_\eps(r,\cdot)||_{L^2}^{2/3} \geq 0 $$
We thus can apply the Fatou's lemma to get from \eqref{fepstepsC}
$$ \liminf_{\eps \to 0} \geq -C+\int_0^1 G(\phi(r,\cdot)) \,\frac{\text{d}r}{r} + \int_0^1 ||\partial_r\phi(r,\cdot)||_{L^2}^2 \, r \dx r  $$
as announced.

\paragraph{Remark.} If we choose to set $\delta = \lambda\eps^{1/3}$, where $\lambda$ has to be precised (and it could possibly depend on $\eps$), the expression~\eqref{fepsteps} becomes
\begin{multline} F_\eps(T_\eps) = \frac{1}{\eps} \int_{\Omega_\delta} (|T_\eps(x)-x|-|T(x)|+|x|) f(x) \dx x + \int_{\Omega_\delta} |DT_\eps|^2 - K\log\lambda \\
+ \frac{1}{\eps} \int_{\Omega\setminus\Omega_\delta} (|T_\eps(x)-x|-|T(x)|+|x|) f(x) \dx x - B\int_\delta^1 ||\psi_\eps(r,\cdot)||_{L^2}^{2/3} \, r \dx r \\
+ \int_\delta^1 (H(\phi_\eps(r,\cdot),\psi_\eps(r,\cdot))-K+B||\psi_\eps(r,\cdot)||_{L^2}^{2/3}) \frac{\text{d}r}{r}+\int_\delta^1 ||\partial_r\phi_\eps(r,\cdot)||_{L^2}^2 \, r \dx r \label{fepstepster} \end{multline}
By using the above estimates \eqref{deuxtiers} and \eqref{xpsi2}, we get that the second line is this time bounded from below~by
$$ A (\inf f) X_\eps - \frac{B'}{\lambda} (X_\eps)^{1/3} $$
which is itself bounded from below by $-C_\lambda = \inf\{ A(\inf f )X-\frac{B'}{\lambda} X^{1/3}\}$; we notice that $C_\lambda$ goes to $0$ as $\lambda\to+\infty$. Let us now compute the first line of \eqref{fepstepster}:
\begin{multline*} \frac{1}{\eps} \int_{\Omega_\delta} (|T_\eps(x)-x|-|T(x)|+|x|) f(x) \dx x + \int_{\Omega_\delta} |DT_\eps|^2 - K\log\lambda
\\ = \frac{\delta^2}{\eps} \int_{\Omega} (|U_\eps(y)-\delta y|-|U_\eps(y)|+|\delta y|) f(\delta y) \dx y + \int_\Omega |DU_\eps|^2 -K\log\lambda \end{multline*}
where we have set $x=\delta y$ and $U_\eps(y) = T_\eps(\delta y)$. Now we use the following expansion
$$ |U_\eps(y)-\delta y|-|U_\eps(y)| = \delta \left(|y|-y\cdot \frac{U_\eps(y)}{|U_\eps(y)|}\right)+o(\delta^2) $$
and recall that $\delta = \lambda \eps^{1/3}$, to get
\begin{multline*} \frac{\delta^2}{\eps} \int_{\Omega} |U_\eps(y)-\delta y|-|U_\eps(y)|+|\delta y| f(\delta y) \dx y + \int_\Omega |DU_\eps|^2 -K\log\lambda \\ 
= \int_\Omega \lambda^3\left(|y|-y\cdot \frac{U_\eps(y)}{|U_\eps(y)|}\right) f(\delta y) \dx y + \int_\Omega |DU_\eps|^2-K\log\lambda + o\left(\lambda^4\eps^{4/3}\right)  \end{multline*}
By choosing, for instance, $\lambda = \eps^{-1/4}$, we get that $\lambda^4\eps^{4/3}\to 0$ and $\lambda^3 f(\delta y) \sim \lambda^3 f(0)$ as $\eps\to 0$. Formally, the leading term of the last expansion is then bounded from below by
$$ - C'_\lambda = \inf\left\{ f(0) \int_{\Omega_1} \lambda^3\left(|y|-y\cdot \frac{U(y)}{|U(y)|}\right)\dx y + \int_{\Omega_1} |DU|^2-K\log\lambda\;:\; U:\Omega_1\to\Omega' \right\}$$
(note the use of the domain $\Omega_1$, which is nothing but a rescaling of the small domain $\Omega_\delta$ close to the origin).
Recalling that $C_\lambda \to 0$, we then claim, since the third line of \eqref{fepstepster} is lower semi-continuous,
$$ \liminf\limits_{\eps \to 0} F_\eps \geq F(T)-\lim_{\lambda\to+\infty} C'_\lambda. $$
We actually conjecture that the $\Gamma$-limit is exactly of this form, possibly modifying the definition of $C'_\lambda$ adding additional constraints on $U$. Indeed, the condition $U\in \Omega'$ is of course necessary, but we can also expect that conditions on the outer boundary of $\Omega_1$, i.e. for $|y|=1$, could be imposed, so as to glue the behavior of $U$ on $\Omega_1$, which corresponds to the behavior of $T$ on $\Omega_\delta$ with the rest of the domain $\Omega$. In particular we expect that the result should be obtained by adding the boundary condition $U(y)=\Phi(\theta)$ for every $y$ on the unit circle, where $\theta$ denotes the angle of $y$ in polar coordinates. These considerations support the conjecture that we mentioned just after the statement of Theorem~\ref{maintheo}.

\subsection{Construction of family of transport maps with equi-bounded energy \label{construction}}

The last point of the proof of Theorem \ref{maintheo} consists in building a family of maps $(T_\eps)_\eps$ such that $(F_\eps(T_\eps))_\eps$ is bounded. The sketch of the proof is the following: starting from a fixed transport map $T = \phi \, \dfrac{x}{|x|}$ satisfying with $\phi(0,\cdot) = \Phi$ (that we call ``the original $T$'' in the following) and which is regular enough except around the origin, we build each $T_\eps$ by modifying $T$ only on $\Omega_\delta$.

\paragraph{Step 1: construction of the original transport map.} We set $T(x) = \phi(r,\theta) \dfrac{x}{|x|}$, where $\phi$ is built as follows:
\begin{itemize}
\item $\phi(0,\theta) = \Phi(\theta)$, and $\phi(\cdot,\theta)$ is increasing and sends the one-dimensional measure $\mu_\theta$ (the starting measure $\mu$ concentrated on the transport ray with angle $\theta$) onto $\nu_\theta/2$ (where $\nu_\theta$ is the target measure on the same transport ray), until the radius $\rho_1$ such that $\phi(\rho_1,\theta) = R_2(\theta)$;
\item starting from this radius $\rho_1$, $\phi(\cdot,\theta)$ is decreasing with the same source and target measure, until the radius $\rho_2$ such that, again, $\phi(\rho_2,\theta) = \Phi(\theta)$. Therefore, on the interval $(\rho_1,\rho_2)$, $\phi(\cdot,\theta)$ sends $\mu_\theta$ onto $\nu_\theta|_{(\Phi(\theta),R_2(\theta))}$;
\item on the last interval (if it is non-empty, which corresponds to $\Phi(\theta) > R_1(\theta)$), $\phi$ is still decreasing and sends $\mu_\theta$ onto $\nu_\theta|_{(R_1(\theta),\Phi(\theta)}$
\end{itemize}

\noindent Precisely, we fix $\theta$ and the expressions of $\mu_\theta$, $\nu_\theta$ are
$$ \dx\mu_\theta(r) = rf(r,\theta) \dx r \qquad \text{and} \qquad \dx\nu_\theta(r) = rg(r,\theta) \dx r $$
which have both same mass on $(0,1)$ and $(R_1(\theta),R_2(\theta))$ respectively. Now we define successively $\rho_1(\theta)$ and $\rho_2(\theta)$ by
$$ \int_{0}^{\rho_1(\theta)} \dx\mu_\theta = \int_{\Phi(\theta)}^{R_2(\theta)} \frac{1}{2}\dx\nu_\theta \qquad \text{and} \qquad \int_{\rho_1(\theta)}^{\rho_2(\theta)} \dx\mu_\theta = \int_{\Phi(\theta)}^{R_2(\theta)} \frac{1}{2}\dx\nu_\theta$$
which are proper definitions thanks to the intermediate value theorem, and imply
$$ \int_{\rho_2(\theta)}^1 \dx\mu_\theta = \int_0^{\Phi(\theta)} \dx\nu_\theta $$
Thus, we have the equality between masses:
$$ \mu_\theta(0,\rho_1(\theta)) = \mu_\theta(\rho_1(\theta),\rho_2(\theta)) = \frac{1}{2}\nu_\theta(\Phi(\theta),1)  \quad \text{and} \quad \mu_\theta(\rho_2(\theta),1) = \nu_\theta(0,\Phi(\theta)) $$
and the measures $\mu_\theta$, $\nu_\theta$ are absolutely continuous on these intervals. We now define the function $\phi(\cdot,\theta)$ as being:
\begin{itemize}
\item on the interval $(0,\rho_1(\theta))$, the unique increasing map $(0,\rho_1(\theta) \to (\Phi(\theta),1)$ sending $\mu_\theta$ onto $\frac{1}{2}\nu_\theta$;
\item on the interval $(\rho_1(\theta),\rho_2(\theta))$, the unique decreasing map $(\rho_1(\theta),\rho_2(\theta)) \to  (\Phi(\theta),1)$ sending $\mu_\theta$ onto $\frac{1}{2}\nu_\theta$;
\item on the interval $(\rho_2(\theta),1)$ (if this interval is not empty), the unique decreasing map $(\rho_2(\theta),1) \to (0,\Phi(\theta))$ sending $\mu_\theta$ to $\nu_\theta$
\end{itemize}
It is easy to check that $\phi(\cdot,\theta)$, defined on the whole $(0,1)$, sends globally $\mu_\theta$ onto $\nu_\theta$. As a consequence, the two-dimensional valued function
$$ T : x=(r,\theta) \in \Omega \mapsto \phi(r,\theta)\frac{x}{|x|} \in \Omega'  $$
is a transport map from $\mu$ to $\nu$.

\paragraph{Step 2: estimates on $\phi$ around the origin.} As above, the principle consists in modifying $T$ on $\Omega_\delta=\Omega \cap B(0,\delta)$, with $\delta = \eps^{1/3}$. In order to obtain again a transport map from $\mu$ to $\nu$, we will have to build a map $S: \Omega_\delta \to T(\Omega_\delta)$ satisfying image-measure constraints; moreover, we also will must guarantee enough regularity on~$S$. For this last point, the following properties of~$T$ will be useful:

\begin{prop} There exists some positive constants $c,C$ depending only on $\Omega'$, $f$, $g$ such that, for $r$ small enough:
\begin{itemize}
\item for any $\theta\in (0,\pi/2)$, 
\begin{equation}  cr^2 \leq \phi(r,\theta)-\Phi(\theta) \leq C r^2  \label{alphadelta1} \end{equation}
\item the function $\phi(r,\cdot)$ is Lipschitz, and
\begin{equation} \Lip (\phi(r,\cdot)-\Phi) \leq C r^2 \label{alphadelta2} \end{equation}
\end{itemize}\end{prop}


\begin{proof} From the fact that $\Phi \leq \sup R_1$ and from the definition of $\rho_1$, we have
$$ \frac{\rho_1(\theta)^2}{2} \sup f \geq \int_0^{\rho_1(\theta)} s f(r,\theta) \dx s = \frac{1}{2}\int_{\Phi(\theta)}^{R_2(\theta)} sg(s,\theta)\dx s \geq \frac{(\inf R_2)^2-(\sup R_1)^2}{4} \inf g $$
which implies that $\rho_1$ is bounded from below by a positive constant $\delta_0$; we will prove the estimates that we want for $r\leq \delta_0$.

First of all, for any $r \leq \delta_0$ (in particular, such an $r$ is smaller than any $\rho_1(\theta)$), the function $\phi(r,\cdot)$ satisfies
$$ \frac{1}{2} \int_{\Phi(\theta)}^{\phi(r,\theta)} \dx\nu_\theta = \int_0^r \dx\mu_\theta $$
Since the respective densities of $\mu_\theta$, $\nu_\theta$ are $s\mapsto sf(s,\theta)$ and $s\mapsto sg(s,\theta)$, we infer
$$ (\inf g) (\inf R_1) (\phi(r,\theta)-\Phi(\theta)) \leq \int_{\Phi(\theta)}^{\phi(r,\theta)} \dx\nu_\theta \leq (\sup g)(\sup R_2) (\phi(r,\theta)-\Phi(\theta))  $$
$$ \text{and} \qquad \frac{r^2}{2}\inf f \leq \int_0^r \dx\mu_\theta \leq \frac{r^2}{2} \sup f   $$
These inequalities immediately leads to \eqref{alphadelta1}.

Let us now introduce the following functions, which are the cumulative distribution functions of $\mu_\theta$ and~$\nu_\theta/2$:
$$ \tilde F(r,\theta) = \int_0^r \dx\mu_\theta = \int_0^r sf(s,\theta)\dx s \qquad \text{and} \qquad \tilde G(r,\theta) = \frac{1}{2}\int_{R_1(\theta)}^r \dx\nu_\theta = \frac{1}{2}\int_{R_1(\theta)}^r sg(s,\theta)\dx s$$
We notice that, for $r\leq \delta_0$, the function $\theta\mapsto \phi(r,\theta)$ is defined by
\begin{equation} \tilde F(r,\theta) = \tilde G(\phi(r,\theta),\theta) - \tilde G(\Phi(\theta),\theta) \label{phirtheta} \end{equation}
In other words, $\phi(r,\theta)$ is the image by the inverse of the map $R\mapsto \tilde G(R,\theta)$ of the point $F(r,\theta)+G(\Phi(\theta),\theta)$. But, from the definition of $G$, we infer that it is $C^1$ w.r.t.~its first variable and its derivative is bounded from below; moreover, all the maps that we consider are Lipschitz w.r.t.~$\theta$. By the inverse function theorem, we deduce that $\phi(r,\cdot)$ is not only well-defined but also Lipschitz. Then, computing the derivative of the equality~\eqref{phirtheta} (which is possible for a.e.~$\theta$ since all the functions are Lipschitz w.r.t.~$\theta$) gives
$$ \partial_2 \tilde F(r,\theta) = \partial_1 \tilde G(\phi(r,\theta),\theta) \partial_2 \phi(r,\theta) + \partial_2\tilde G (\phi(r,\theta),\theta) - \partial_1 \tilde G(\Phi(\theta),\theta) \Phi'(\theta) - \partial_2 \tilde G(\Phi(r,\theta),\theta) $$
which leads to (we omit the dependence of $\phi$, $\Phi$ on $r,\theta$ for the simplicity of notations):
\begin{equation} \partial_1\tilde G (\phi,\theta) (\partial_2\phi - \Phi') = \partial_2\tilde F (r,\theta) + \partial_1 \Phi' (\partial_1 \tilde G (\Phi,\theta) - \partial_1 \tilde G (\phi,\theta)) - (\partial_2 \tilde G (\phi,\theta) - \partial_2 \tilde G (\Phi,\theta)) \label{partialtildeG} \end{equation}
By using the definitions of $\tilde F$, $\tilde G$, we now notice that:
\begin{itemize}
\item first, $ 0 \leq  \partial_2 \tilde F (r,\theta) = \displaystyle\int_0^r s\partial_2 f(s,\theta) \dx s \leq \dfrac{r^2}{2} \Lip f $ ;
\item second, since $\partial_1 \tilde G (r,\theta) = \dfrac{1}{2}r g(r,\theta)$ and from the fact that $g$ and $\Phi$ are Lipschitz,
$$ |\partial_1 \Phi' (\partial_1 \tilde G (\Phi,\theta) - \partial_1 \tilde G (\phi,\theta))| \leq \frac{1}{2} (\Lip \Phi)(\sup R_2)(\Lip g)|\Phi(\theta)-\phi(r,\theta)|  $$
which is controlled by $r^2$ thanks to the estimate \eqref{alphadelta1};
\item third, $\partial_2 \tilde G(r,\theta) = \dfrac{1}{2} \displaystyle\int_{R_1(\theta)}^r s\partial_2 g(s,\theta) \dx s - \dfrac{1}{2} R_1'(\theta) R_1(\theta) g(R_1(\theta),\theta)  $ thus
$$ |\partial_2 \tilde G (\phi,\theta) - \partial_2 \tilde G (\Phi,\theta)| \leq \frac{1}{2} \int_{\Phi(\theta)}^{\phi(r,\theta)} s|\partial_2 g(s,\theta)|\dx s \leq \frac{1}{2} (\sup R_2)(\Lip g)(\phi(r,\theta)-\Phi(\theta))  $$
which is, again thanks to the estimate \eqref{alphadelta1}, controlled by $r^2$.
 \end{itemize}
By inserting the three above estimates in \eqref{partialtildeG}, and keeping in mind that $\partial_1 G$ is bounded from below, it is clear that we get $|\partial_\theta (\phi(r,\theta)-\Phi(\theta))| \leq C r^2$, for any $\theta$ where this derivative exists; this achieves the proof of the estimate \eqref{alphadelta2}.
\end{proof}

\paragraph{Step 3: perturbation of the optimal $T$.} In what follows, we denote by
$$ \Omega_\delta = \Omega_1 \cap B(0,\delta) = \{x=(r,\theta) : 0 < r < \delta \text{ and } 0 < \theta < \pi/2\} $$
$$ \text{and} \qquad \Omega'_\delta = T(\Omega_\delta) = \left\{x=(r,\theta) : \Phi(\theta) < r < \phi(\delta,\theta) \text{ and } 0 < \theta < \pi/2 \right\} $$
We now denote by:
\begin{itemize}
\item $S_1 : x \in \Omega_\delta \mapsto \dfrac{x}{\delta} \in \Omega_1$ and $f_\delta(x) = f(\delta x)$. Notice that $f_\delta = \dfrac{1}{\delta^2} (S_1)_\# (f|_{\Omega_\delta})$;
\item $\Omega_2$ is the rectangle $(0,1) \times (0,\pi/2)$ and
$$ S_2 : (\lambda,\theta) \in \Omega_2 \mapsto x = (\Phi(\theta)+\lambda(\phi(\delta,\theta)-\Phi(\theta)),\theta) \in \Omega_\delta' $$
where $x$ is here written in polar coordinates. We also denote by $g_\delta = \dfrac{1}{\delta^2} \left((S_2)^{-1}\,_\# g|_{\Omega_\delta'}\right)$.
\end{itemize}
As above, $S_1$ and $S_2$ actually depend on $\delta$ but we omit the index $\delta$ for the sake of simplicity of notations. We have of course $\inf f \leq f_\delta(x) \leq \sup f$ for any $x \in \Omega_1$, and $\Lip f_\delta \leq \delta \Lip f$. On the other hand, since $S_2$ is Lipschitz and one-to-one, the Monge-Amp\`ere equation provides
$$ \det DS_2(\lambda,\theta) = \frac{g_\delta(\lambda,\theta)}{g(S_2(\lambda,\theta))} $$
\begin{equation} \text{thus} \qquad g_\delta(\lambda,\theta) = (\phi(\delta,\theta)-\Phi(\theta)) \frac{1}{\delta^2} g(S_2(\lambda,\theta)) \label{gdelta} \end{equation}
We deduce from \eqref{alphadelta1} and \eqref{gdelta} that $c \leq g_\delta \leq C$ for $c,C$ positive and independent of $\delta$. Moreover, \eqref{gdelta} provides
$$ \Lip g_\delta \leq \sup (\phi_\delta-\phi) \frac{1}{\delta^2} \Lip (g\circ S_2) + \Lip(\phi_\delta-\Phi) \frac{1}{\delta^2} \sup g $$
Again, thanks to \eqref{alphadelta1} and \eqref{gdelta} and using that $\Lip S_2$ is uniformly bounded with respect to $\delta$, we get $\Lip g_\delta \leq C$ independent of~$\delta$.

We now claim that there exists two bi-Lipschitz diffeomorphisms $\psi_1$, $\psi_2$ sending $\bar{\Omega_1}$, $\bar{\Omega_2}$ to the unit ball and satisfying the condition $(\psi_2)^{-1}\circ \psi_1(x) = \theta$ while $x$ has $(1,\theta)$ as polar coordinates. The construction that we propose is the following (see Figure~\ref{fig2} below):
\begin{itemize}
\item apply translation and rotation to transform $\Omega_1$ into the domain $U_1$ below;
\item consider the rectangle $U_2$, whose ``upper-right'' and ``lower-right'' corners coincide with the corners of $U_1$: applying Lemma~\ref{lemdiffeo} gives two maps $\alpha$, $\beta$ from $U_1$, $U_2$ to the unit ball which are Lipschitz, with Lipschitz determinant and preserve angles: in other words, any point of the right-boundary of $U_1$ with angle $\theta'$ will be sent by $\beta^{-1}\circ\alpha$ to a point of the right vertical boundary of $U_2$ with same angle $\theta'$. For $x \in \Omega_1$, we call then $\psi_1(x)$ the image by $\alpha$ of the corresponding point of $U_1$;
\item moreover, it is clear that the maps which associates to any $\theta \in (0,1)$ the angle $\theta'$ of the point of $U_1$ corresponding to the point $(1,\theta)$ of $\Omega_1$, is $C^\infty$ with $C^\infty$ inverse. We may then consider a new map $U_2 \mapsto \Omega_2$, which is affine w.r.t.~the horizontal coordinate, and so that any element of the right-side of $U_2$ with angle $\theta'$ is mapped to the element of the right-side of $\Omega_2$ with vertical coordinate $\theta$. It is clear that this map, denoted by $\tilde\beta$ is a $C^\infty$-diffeomorphism; denoting by $\psi_2 = \tilde\beta^{-1}\circ\beta$, we get the required results, namely that $\psi_2$ is bi-Lipschitz with Lipschitz Jacobian determinant and, by construction, the boundary condition on $(\psi_2)^{-1}\circ\psi_1$ is satisfied.
\end{itemize}

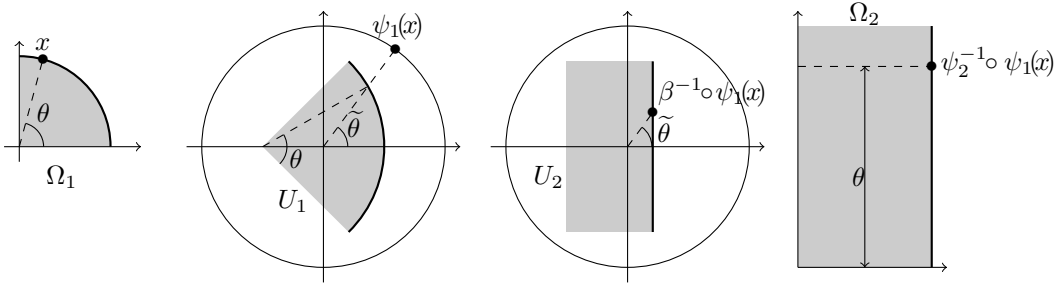
\begin{figure}[h] 
\begin{center}
\begin{tikzpicture}[xmin=-5.5,xmax=8,ymin=-2,ymax=-2,scale=0.8]

  \fill[gray!40] (-3.5,0) arc (0:90:1.5) -- (-5,0) -- cycle ;
  \draw[thick] (-3.5,0) arc (0:90:1.5) ;
  \draw [->] (-5.25,0) -- (-3,0) ;
  \draw [->] (-5,-0.25) -- (-5,1.75) ;
  \draw[dashed] (-5,0) -- (-4.61177143235,1.44888873943) ;
  \draw (-4.61177143235,1.44888873943) node {\small $\bullet$} ;
  \draw (-4.61177143235,1.44888873943) node[above] {$x$};
  \draw (-4.6,0.25) node[above]{$\theta$};
  \draw (-4.6,0) arc (0:75:0.4) ;
  \draw (-4.3,-0.5) node {$\Omega_1$};

  \draw (0,0) circle (2cm) ;
  \fill[gray!40] (0.414213562,-1.414213562) arc (-45:45:2) --  (-1,0) -- cycle ;
  \draw [->] (-2.25,0) -- (2.25,0) ;
  \draw [->] (0,-2.25) -- (0,2.25) ; 
  \draw[dashed] (-1,0) -- (0.732050808,1) ;
  \draw[dashed] (0,0) -- (1.18138100675,1.61379646582) ;
  \draw (1.18138100675,1.61379646582)  node {\small $\bullet$} ;
  \draw (1.23,1.62) node[above] {$\psi_1\!(\!x\!)$};
  
  \draw[thick] (0.414213562,1.414213562) arc (45:-45:2) ;
  \draw (-0.71715728752,-0.28284271247) arc (-45:30:0.4) ;
  \draw (-0.45,-0.5) node[above]{$\theta$};
  \draw (0.4,0) arc (0:53.7939689:0.4) ;
  \draw (0.5,0) node[above]{$\tilde\theta$};
  \draw (-0.5,-0.9) node {$U_1$};

  \draw (5,0) circle (2cm) ;
  \fill[gray!40] (4,-1.414213562) rectangle (5.414213562,1.414213562) ;
  \draw[thick] (5.414213562,1.414213562) -- (5.414213562,-1.414213562) ;
  \draw [->] (2.75,0) -- (7.25,0) ;
  \draw [->] (5,-2.25) -- (5,2.25) ; 
  \draw[dashed] (5,0) -- (5.414213562,0.565826248) ;
  \draw (5.414213562,0.565826248) node {\small $\bullet$} ;
  \draw (6.4,0.9) node {$\beta^{-1}\!\!\circ\!\psi_1\!(\!x\!)$} ;
  \draw (5.4,0) arc (0:53.7939689:0.4) ;
  \draw (5.6,0.3) node {$\tilde\theta$};
  \draw (3.7,-0.5) node {$U_2$};

  \fill[gray!40] (7.8,-2) -- (10,-2) -- (10,2) -- (7.8,2) -- cycle;
  \draw [->] (7.8,-2) -- (10.25,-2) ;
  \draw [->] (7.8,-2) -- (7.8,2.25) ;
  \draw[dashed] (7.8,1.33) -- (10,1.33);
  \draw (10,1.33) node {\small$\bullet$} ;
  \draw (11.1,1.4) node {$\psi_2^{-1}\!\!\circ\psi_1\!(\!x\!)$} ;
  \draw [<->] (8.9,-2) -- (8.9,1.33) ;
  \draw (8.8,-0.5) node {$\theta$};
  \draw[thick] (10,-2) -- (10,2) ;
  \draw (8.9,2.2) node {$\Omega_2$};

\end{tikzpicture}
\begin{minipage}{14.5cm} \caption{The domains $\Omega_1$, $U_1$, $U_2$, $\Omega_2$. The point with angle $\theta$ is mapped onto the point of $U_1$ with angle~$\tilde\theta$. The Lemma~\ref{lemdiffeo} gives us angle-preserving diffeomorphisms, so the corresponding point of $U_2$ has also $\tilde\theta$ as angle. It is mapped onto the point of $\Omega_2$ of vertical coordinate $\theta$.} \label{fig2} \end{minipage}
\end{center}
\end{figure}

\noindent We may then apply Corollary~\ref{codm}, with the densities $f_\delta$, $g_\delta$ on the domains $\Omega_1$, $\Omega_2$, to get a map $U_\delta: \Omega_1 \mapsto \Omega_2$ satisfying the statement of the Corollary. In particular, thanks to the bounds on $f_\delta$, $g_\delta$ and on their Lipschitz constants and, the Lipschitz constant of $U_\delta$ is bounded uniformly on $\delta$ by some constant $C$. Now we consider $S_\delta := S_2\circ U_\delta \circ S_1$. Given the image measures of $\mu$, $f_\delta$, $g_\delta$ by the maps $S_1$, $U_\delta$, $S_2$, it is clear that $S_\delta$ sends $\mu|_{\Omega_\delta}$ to $\nu|_{\Omega_{\delta}'}$. Moreover, we have the following estimates:
$$ \Lip U_\delta \leq C, \quad \Lip S_2 \leq \Lip \Phi \quad \text{and } \Lip S_1 = \frac{1}{\delta} $$
$$ \text{thus} \qquad \Lip S_\delta \leq \frac{C}{\delta} $$
for some constant $C$ which does not depend on $\delta$. Finally, given the expression of $S_2$ and the boundary condition satisfied by $U_\delta$, it is clear that $S_\delta(x) = T(x)$ for $|x| = \delta$. Therefore, the map $T_\eps$ defined by
$$ T_\eps(x) = \left\{ \begin{array}{ll} T(x) & \text{if } |x| \geq \delta \\ S_\delta(x) & \text{if } x \in \Omega_\delta \end{array} \right. $$
is globally Lipschitz on $\Omega$.

\paragraph{Step 4: estimates on $F_\eps(T_\eps)$.} We restart from the expression \eqref{fepsteps}, and use the facts that $\psi_\eps = 0$ and that $T_\eps=T$ outside of $\Omega_\delta$:
\begin{multline*} F_\eps(T_\eps) = \frac{1}{\eps} \int_{\Omega_\delta} (|S_\delta(x)-x|-|S_\delta(x)|+|x|) f(x) \dx x + \int_{\Omega_\delta} |DS_\delta(x)|^2 \dx x \\ 
+ \int_\delta^1 (||\phi(r,\cdot)||^2_{L^2}+||\partial_\theta\phi(r,\cdot)||^2_{L^2}-K)\,\frac{\text{d}r}{r} + \int_\delta^1 ||\partial_r\phi(r,\cdot)||^2_{L^2} \, r \dx r \end{multline*} 
We still have $|S_\delta(x)-x|-|S_\delta(x)|+|x| \leq 2|x|$ and $|DS_\delta(x)| \leq C/\delta$, so that
$$ \frac{1}{\eps} \int_{\Omega_\delta} (|S_\delta(x)-x|-|S_\delta(x)|+|x|) f(x) \dx x + \int_{\Omega_\delta} |DS_\delta(x)|^2 \dx x \leq \frac{\pi \sup f}{3} \frac{\delta^3}{\eps} + \frac{C^2\pi}{4} $$
which is bounded since $\delta = \eps^{1/3}$. On the other hand,
$$ ||\phi(r,\cdot)||_{L^2}^2+||\partial_\theta\phi(r,\cdot)||_{L^2}^2-K  = (||\phi(r,\cdot)||_{L^2}^2-||\Phi||_{L^2}^2)+(||\partial_\theta\phi(r,\cdot)||_{L^2}^2-||\Phi'||_{L^2}^2) $$
$$  = \langle \phi(r,\cdot)-\Phi,\phi(r,\cdot)+\Phi \rangle_{L^2} + \langle \partial_\theta\phi(r,\cdot)-\Phi',\partial_\theta\phi(r,\cdot)+\Phi' \rangle_{L^2} $$
$$ \leq ||\phi(r,\cdot)-\Phi'||_{L^1} (||\phi(r,\cdot)||_\infty+||\Phi||_\infty) + ||\partial_\theta\phi(r,\cdot)-\Phi'||_{L^1} (\Lip \phi+\Lip \Phi) $$
Since $\Phi$, $\phi(r,\cdot)$ are valued in $\Omega'$, their $L^\infty$-norm are controlled by $\sup R_2$. By combining this and the estimates \eqref{alphadelta1} and \eqref{alphadelta2}, we obtain
$$ 0 \leq ||\phi(r,\cdot)||_{L^2}^2+||\partial_\theta\phi(r,\cdot)||_{L^2}^2-K \leq Cr^2 $$
for $r$ small enough (and where $C$ does not depend on $r$). On the other hand, we know that $\phi$, $\Phi$ and their derivatives are globally bounded on $(0,1)\times(0,\pi/2)$. This proves that
$$ \int_0^1 (||\phi(r,\cdot)||^2_{L^2}+||\partial_\theta\phi(r,\cdot)||^2_{L^2}-K)\,\frac{\text{d}r}{r} < +\infty $$
$$ \text{and} \qquad \int_0^1 ||\partial_r\phi(r,\cdot)||^2_{L^2} \, r \dx r < +\infty $$
and we conclude that $(F_\eps(T_\eps))_\eps$ is bounded, as required.


\bibliographystyle{acm}
\bibliography{dpls}

\end{document}